\documentclass[a4paper]{amsart}

\usepackage{amsmath}
\usepackage{mathtools}
\usepackage[all]{xy}

\usepackage[latin1]{inputenc}
\usepackage[T1]{fontenc}
\usepackage{varioref}
\usepackage{amsfonts}
\usepackage{amssymb}
\usepackage{bbm}

\usepackage[american]{babel}

\usepackage{graphicx}

\theoremstyle{definition}
\newtheorem{defin}{Definition}[section]
\newtheorem{ex}[defin]{Example}
\newtheorem*{Ack}{Acknowledgements}

\theoremstyle{plain} 
\begingroup 
\newtheorem{thm}[defin]{Theorem}
\newtheorem{prop}[defin]{Proposition}
\newtheorem{lem}[defin]{Lemma}

\newtheorem{conj}[defin]{Conjecture}

\endgroup

\theoremstyle{remark}
\newtheorem*{rem}{Remark}
\newtheorem*{mind}{Reminder}
\newtheorem*{caut}{Caution}

\newcommand{\lieg}{\mathfrak{g}}

\newcommand{\liesl}{\mathfrak{sl}}
\newcommand{\liesp}{\mathfrak{sp}}
\newcommand{\lieso}{\mathfrak{so}}

\newcommand{\kxC}{\mathbbm{C}}

\newcommand{\prP}{\mathbbm{P}}
\newcommand{\intZ}{\mathbbm{Z}}

\newcommand{\II}{\mathcal{I}}

\newcommand{\midd}{\hspace{-0.35ex} \mid \hspace{-0.35ex}}
\newcommand{\mi}{\hspace{-0.6ex}-\hspace{-0.3ex}}
\newcommand{\pl}{\hspace{-0.4ex}+\hspace{-0.3ex}}

\newcommand{\plu}{\hspace{-0.6ex}+\hspace{-0.3ex}}

\newcommand{\Ad}{\operatorname{Ad}}

\newcommand{\Bl}{\operatorname{Bl}}
\newcommand{\codim}{\operatorname{codim}}

\newcommand{\disc}{\operatorname{disc}}
\newcommand{\elm}{\operatorname{elm}}

\newcommand{\Grass}{\operatorname{Grass}}
\newcommand{\Hom}{\operatorname{Hom}}
\newcommand{\Hilb}{\operatorname{Hilb}}

\newcommand{\im}{\operatorname{im}}
\newcommand{\Mat}{\operatorname{Mat}}

\newcommand{\Pf}{\operatorname{Pf}}

\newcommand{\rk}{\operatorname{rk}}
\newcommand{\sgn}{\operatorname{sgn}}
\newcommand{\Spec}{\operatorname{Spec}}
\newcommand{\Sym}{\operatorname{Sym}}
\newcommand{\tr}{\operatorname{tr}}

\newcommand{\red}{\mathbin{\textit{\hspace{-0.8ex}/\hspace{-0.5ex}/\hspace{-0.4ex}}}}
\newcommand{\sred}{\mathbin{\textit{\hspace{-0.8ex}/\hspace{-0.5ex}/\hspace{-0.5ex}/\hspace{-0.4ex}}}}

\newcommand{\vertsub}{\mathrel{\protect\raisebox{0.3ex}{\protect\rotatebox[origin=c]{90}{$\subset$}}}}
\newcommand{\vertin}{\mathrel{\protect\raisebox{0.3ex}{\protect\rotatebox[origin=c]{90}{$\in$}}}}

\newcommand{\acts}{\mathrel{\protect\raisebox{0.3ex}{\protect\rotatebox[origin=c]{270}{$\circlearrowright$}}}}

\author{Tanja Becker}

\title{On the existence of symplectic resolutions of symplectic reductions}
\address{Fachbereich Physik, Mathematik und Informatik\\ Johannes Gutenberg -- Universit\"at Mainz \\ Staudinger Weg 9 \\ D -- 55099 Mainz}

\email{tanja@mathematik.uni-mainz.de}

\begin{document}

\begin{abstract}
We compute the symplectic reductions for the action of $Sp_{2n}$ on several copies of $\kxC^{2n}$ and for all coregular representations of $Sl_2$. If it exists we give at least one symplectic resolution for each example. In the case $Sl_2$ acting on $\liesl_2 \oplus \kxC^2$ we obtain an explicit description of Fu's and Namikawa's example of two non-equivalent symplectic resolutions connected by a Mukai flop \cite{funa:2004}.
\end{abstract}

\maketitle

\section{introduction}

Let $G$ be a semisimple complex linear algebraic group, $\lieg$ its Lie algebra and $V$ a finite dimensional representation of $G$. The double $V \oplus V^*$ is equipped with a symplectic form and $G$ acts symplectically on this double. Let $\mu \colon V \oplus V^* \to \lieg^*$ be the associated moment map. Define the symplectic reduction by
$$V\oplus V^* \sred G := \mu^{-1}(0)\red G.$$

\begin{conj}[Kaledin, Lehn, Sorger] 
\label{KLSVerm} 
If there exists a symplectic resolution of $V\oplus V^* \sred G$, then the quotient $V\red G$ of the simple action is smooth.
\end{conj}

The analogue for finite groups is known and has been proved by Kaledin \cite{kal:2003} and Verbitsky \cite{ver:2000}.

Representations with smooth quotient $V\red G$ are called coregular. They have been classified by Schwarz in \cite{schw:1978} in case $G$ is connected, simple and simply connected.
This complete classification suggests to examine the converse of conjecture \ref{KLSVerm}: Let $G$, $V$ be in Schwarz' list. Does there exist a symplectic resolution of $V\oplus V^*\sred G$? Note that for $G$ finite, the converse of the conjecture is not true by Ginzburg--Kaledin \cite{gika:2004}. However, for $G = Sl_2 := Sl_2(\kxC)$ we obtain the following result:

\begin{thm}
Let $V$ be a coregular representation of $Sl_2$. Then every irreducible component of $V \oplus V^*\sred Sl_2$ with reduced structure is a symplectic variety in the sense of Beauville--Namikawa \cite{bea:2000} and admits a symplectic resolution.
\end{thm}

More precisely, from Schwarz' list we see that there are eight coregular representations of $Sl_2$ which can be subdivided into four different types:
\begin{enumerate}
\item $\kxC^2$, $\kxC^2 \oplus \kxC^2$ and $\kxC^2 \oplus \kxC^2 \oplus \kxC^2$,
\item $S^3\kxC^2$ and $S^4\kxC^2$,
\item $\liesl_2$ and $\liesl_2 \oplus \liesl_2$,
\item $\liesl_2 \oplus \kxC^2$.
\end{enumerate}

\begin{rem}
These representations can be found in \cite{schw:1978} as follows.  
Tables $1a$ and $2a$ list all coregular representations of $Sl_n$, but we only consider $n = 2$. In table $1a$, items $1$ and $2$ coincide and give type (1). Items $18$ and $19$ are also the same, they contain type (3). Items $12$, $13$ and $20$ all encode representation (4). The other items do not apply to the case $n = 2$. In table $2a$, items $1$ and $2$ contain two further representations of  $Sl_2$, these are the symmetric powers, type (2). 

In addition, we consider item $1$ of table $4a$, namely a series of representations of $Sp_{2n}$ on at most $2n+1$ copies of $\kxC^{2n}$. Type (1) is a special case of this, so actually we will analyse this more general case instead.
\end{rem}

We begin with the representation of $Sp_{2n}$. To obtain a statement on symplecticity of the quotient a large effort is spent on examining normality of the zero fibre of the moment map:

\begin{thm}
\label{thmnormal}
For the action of $Sp_{2n}$ on $(\kxC^{2n})^{\oplus 2m}$ the zero fibre of the moment map is reduced if $m \geq 2n$ and even normal if $m \geq 2n + 1$.
\end{thm}

From this we deduce that the symplectic reduction actually is a symplectic variety in the extreme case $m = 2n + 1$. By means of nilpotent orbits we then analyse in which case the quotient is symplectically resolvable. 
As $Sl_2$ and $Sp_2$ coincide this covers the first type of representations of $Sl_2$ (1). 

Afterwards we shortly review type (2) and (3): 
Type (2) is classical and has already been treated by Hilbert. The resolutions we find there are well-known. 
Type (3) can be reduced to analysing the action of $SO_3$ on $\kxC^3$, and will be treated similarly to $Sp_{2n}\acts \kxC^{2n}$. 

The most interesting example is the action on $\liesl_2 \oplus \kxC^2$. There the symplectic reduction has a special configuration, which yields different symplectic resolutions: 
\begin{thm} The symplectic reduction $Z := \liesl_2 \oplus \kxC^2 \oplus (\liesl_2 \oplus \kxC^2)^*\sred Sl_2$ admits three non-isomorphic symplectic resolutions $\widetilde Z$, $\widetilde{Y}_1$ and $\widetilde{Y}_2$ connected by Mukai flops on the components $E_1$ resp. $E_2$ of the zero fibre of $\pi \colon \widetilde Z\to Z$:
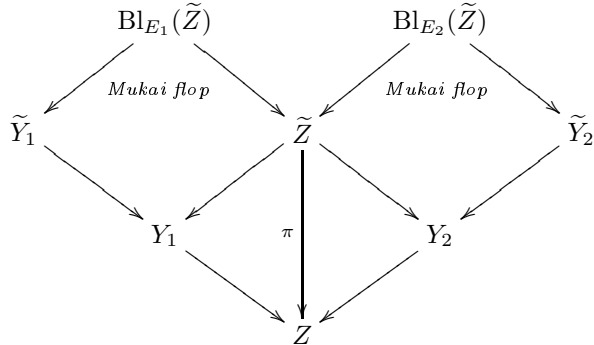
$$ \begin{xy} \xymatrix{
& \Bl_{E_1}(\widetilde Z) \ar@{->}^{\text{\phantom{m}Mukai flop}}[ld] \ar@{->}[rd] & & \Bl_{E_2}(\widetilde Z) \ar@{->}^{\text{\phantom{m}Mukai flop}}[ld] \ar@{->}[rd] & \\
\widetilde{Y}_1 \ar@{->}[rd] & & \widetilde{Z} \ar@{->}[ld] \ar@{->}_{\pi}[dd] \ar@{->}[rd] & & \widetilde{Y}_2 \ar@{->}[ld] \\
& Y_1 \ar@{->}[rd] & & Y_2 \ar@{->}[ld] &\\
& & Z & &
}\end{xy} $$
While the resolutions $\widetilde Y_1$ and $\widetilde Y_2$ are equivalent in the sense of \cite{funa:2004}, $\widetilde Z$ is non-equivalent to them.
\end{thm}
This gives an explicit description of the two non-equivalent symplectic resolutions constructed by Fu and Namikawa in \cite{funa:2004}. Our construction by blowing up and using the cotangent bundle shows that both resolutions are algebraic.

\begin{Ack}
I would like to thank Manfred Lehn for posing the original problem and his excellent supervision of my work. The basic idea for the proof of normality is due to him. Further, I thank Christoph Sorger for his ideas to improve the presentation. I am very greatful to Baohua Fu for several discussions, corrections and the help with the Mukai flop.

While this paper was written I have been partially supported by DAAD and SFB/TR 45, which is greatly acknowledged.
\end{Ack}

\section{Symplectic reductions}
\label{sred}

Before turning to the different examples we give a short introduction to symplectic geometry and work out the aspects we need concerning symplectic resolutions. In order to do so we introduce the moment map, use it to define a modified quotient, and give a criterion when this reduction actually is symplectic.
\vspace{1em}

A normal algebraic variety $X$ over $\kxC$ with symplectic structure $\sigma \in \Gamma(X_{reg},\Omega_{X_{reg}}^2)$ on its regular part is called a \emph{symplectic variety} if for one (hence any) resolution of singularities $f\colon\widetilde{X} \to X$ the section $\sigma$ extends to $\widetilde{X}$. 
A resolution $f$ is called \emph{symplectic} if $\sigma$ extends to a symplectic structure on $\widetilde X$, i.e. is everywhere non-degenerate.

Symplectic resolutions are semismall by Kaledin \cite{kal:2006} and Namikawa \cite{nam:2001}. 
The property of semismallness can also assure symplecticity of certain resolutions:

\begin{prop}
\label{cohalbsympl}
Let $X$ be a variety with symplectic form on $X_{reg}$. If $\pi\colon \widetilde{X} \to X$ is a semismall resolution such that for some closed subset $F$ of $X$ of codimension $\geq 4$ the restriction $\widetilde{X}\setminus \pi^{-1}(F) \to X \setminus F$ is a symplectic resolution, then $\pi$ is a symplectic resolution.
\end{prop}
\begin{proof}
Since $\widetilde{X}$ is smooth, the induced symplectic form on $\pi^{-1}(X_{reg})$ extends to a form $\sigma$ on $\widetilde{X}$ because the bundle $\Omega^2_{X_{reg}}$ extends, and $\sigma$ is symplectic outside $\pi^{-1}(U)$. If the extension was degenerate, its determinant would vanish along a divisor.
But as $\pi$ is semismall, the codimension of $\pi^{-1}(U)$ in $\widetilde{X}$ is at least $2$, and $\sigma$ is not symplectic at most in codimension $2$. So $\det \sigma$ cannot vanish at all.
\end{proof}


Now let us turn to the symplectic double $V \oplus V^*$ of a vector space $V$. Its symplectic structure is given by $\sigma((v, \eta),(w, \zeta)) = \eta(w) - \zeta(v)$ and $G$ acts symplectically on it via $g(v,\eta) = (gv, \eta \circ g^{-1})$. 

If there is a non-degenerate bilinear form $<\hspace{-0.5ex}\cdot,\cdot\hspace{-0.5ex}>\colon V \times V \to \kxC$, this action can be computed in the following way: as a vector space $V$ can be identified with its dual via the induced isomorphism $V \to V^*$, $v \mapsto <\hspace{-0.5ex}\cdot,v\hspace{-0.5ex}>$. Having $v, w \in V$, the dual action $g\cdot w$ is defined via $<\hspace{-0.5ex}v,g\cdot w\hspace{-0.5ex}> =  <\hspace{-0.5ex}\cdot,w\hspace{-0.5ex}>\hspace{-1ex}(g^{-1}v)\hspace{-0.5ex} = <\hspace{-0.5ex}g^{-1}v,w\hspace{-0.5ex}>$
by taking the adjoint of $g^{-1}$  in the last term.
If $<\hspace{-0.5ex}\cdot,\cdot\hspace{-0.5ex}>$ is even $G$--invariant, then $<\hspace{-0.5ex}g^{-1}v,w\hspace{-0.5ex}>\, = \,<\hspace{-0.5ex}gg^{-1}v,gw\hspace{-0.5ex}>\, = \,<\hspace{-0.5ex}v,gw\hspace{-0.5ex}>,$ 
therefore $V = V^*$ as a representation.
\vspace{1em}

We recall the moment map which plays a central role in our considerations. 
Let $(V,\sigma)$ be a symplectic vector space and $G$ a semisimple group with Lie algebra $\lieg$. Equip $\lieg^*$ with the coadjoint action. If $G$ acts symplectically on $V$, there exists a unique $G$--equivariant map $\mu: V \to \lieg^*$ which satisfies $d\mu_x(\xi)(A) = \sigma(\xi,Ax)$ for all $x \in V,\, \xi \in T_xV, \; A \in \lieg$. This map $\mu$ is called the \emph{moment map}.

On the symplectic double of a vector space the moment map turns out to be 
$$\mu\colon V \oplus V^* \longrightarrow \lieg^*,\; (v,\eta) \longmapsto (f\colon A \mapsto \eta(Av)).$$

Let us denote by $\II_{\mu}$ the ideal corresponding to $\mu^{-1}(0)$, i.e. $\mu^{-1}(0) = \Spec(\kxC[V]/\II_{\mu})$.

\begin{rem}
For every symplectic action $G \times V \to V$ with moment map $\mu$ the $G$--equivariance of $\mu$ ensures that $\mu^{-1}(0)$ is a $G$--invariant subset of $V$. This is why there is also an action of $G$ on $\mu^{-1}(0)$ and we can consider its quotient:
\end{rem}

\begin{defin}
Let $V$ be a symplectic vector space with symplectic $G$--action and corresponding moment map $\mu$. The \emph{symplectic reduction} of $V$ is defined to be
$$V\sred G := \mu^{-1}(0)\red G = \Spec(\kxC[V]/\II_{\mu})^G.$$
\end{defin}

\begin{caut} 
In spite of its misleading name, the symplectic reduction need not be a symplectic variety.
\end{caut}

\begin{prop}
\label{satzsymplStr}
Let $G \acts V$ be a symplectic action on an symplectic vector space $V$, let $\mu\colon V \to \lieg^*$ denote the moment map. Then for every $x \in \mu^{-1}(0)$ with closed orbit $Gx \subset \mu^{-1}(0)$ and trivial isotropy group, the image $\overline{x} \in \mu^{-1}(0)\red G$ is a regular point and $T_{\overline{x}} (\mu^{-1}(0)\red G) = T_x \mu^{-1}(0) / \lieg x = (\lieg x)^{\perp} / \lieg x$ is a symplectic vector space.
\end{prop}
\begin{proof}
Luna's slice theorem (cf. \cite{lun:1973}) implies regularity and the first equality. 
The differential $d\mu_x:T_xV \to T_{\mu(x)}\lieg^*$ has maximal rank in all points $x \in \mu^{-1}(0)_{reg}$ and we have $T_x\mu^{-1}(0) = \ker d\mu_x = \{ \xi \in T_xV \mid d\mu_x(\xi)(A) = 0 \; \forall \, A \in T_xV \}$.
Together with the property $d\mu_x(\xi)(A) = \sigma(\xi,Ax)$ of the moment map this implies
$$ T_x\mu^{-1}(0) = \{ \xi \in T_xV \mid \sigma(\xi,Ax) = 0 \; \forall \, A \in \lieg \} = (\lieg x)^{\perp}.$$
This shows the second equality. 
By $G$--invariance, if $x$ is contained in $\mu^{-1}(0)$ so is the whole orbit $Gx$. For this reason $\lieg x$ is contained in $T_x\mu^{-1}(0)$, and it is an isotropic subspace of $T_xV$ because $\lieg x \subset T_x\mu^{-1}(0) = (\lieg x)^{\perp}$. This implies $(\lieg x)^{\perp} / \lieg x$ is a symplectic vector space. In particular there is a symplectic structure on the set of points with closed orbit and trivial isotropy group.
\end{proof}

This implies the following criterion for the existence of a symplectic structure on the symplectic reduction:

\begin{prop}
\label{complint}
$(\mu^{-1}(0)\red G)_{reg}$ possesses a symplectic structure if $\mu^{-1}(0)$ is normal, the isotropy group of every point in $\mu^{-1}(0)_{reg}$ is trivial and for the singular parts $\mu^{-1}(0)_{sing}\red G \subset (\mu^{-1}(0)\red G)_{sing}$ holds.
\end{prop}
\begin{proof} Since $\mu^{-1}(0)_{sing}\red G \subset (\mu^{-1}(0)\red G)_{sing}$ every $\overline{x} \in (\mu^{-1}(0)\red G)_{reg}$ has a representative $x$ with a closed orbit in $\mu^{-1}(0)_{reg}$.
\end{proof}

\section{The action of $Sp_{2n}$ on $(\kxC^{2n})^{\oplus m}$}
\label{KapSp}

Throughout this section we will consider the symplectic group $Sp_{2n}$ with defining matrix $J = \left( \begin{smallmatrix} 0 & I_n \\ -I_n & 0 \end{smallmatrix}\right)$, and its quadratic analogue $Q = \left(\begin{smallmatrix} 0 & I_m \\ I_m & 0 \end{smallmatrix}\right)$.

We examine the symplectic double of the action
$$ \varrho_m^n \colon Sp_{2n} \times (\kxC^{2n})^{\oplus m} \to (\kxC^{2n})^{\oplus m},\; (g, x^{(1)}, \hdots, x^{(m)}) \mapsto (gx^{(1)}, \hdots, gx^{(m)}).$$

To compute its symplectic double remark that the non-degenerate bilinear form $\kxC^{2n} \times \kxC^{2n} \to \kxC$, $(x,y) \mapsto y^tJx$ is $Sp_{2n}$--invariant because $J$ is. Thus we see the representation is self-dual so that the symplectic double of $\varrho_{m}^n$ is $\varrho_{2m}^n$.

Writing $x^{(i)} = \begin{pmatrix} x_{1i} & \hdots & x_{2n,i} \end{pmatrix}^t$ and arranging these vectors as matrices $X' := (x^{(1)} \mid \hdots \mid x^{(m)})$, $X'' := (x^{(m+1)} \mid \hdots \mid x^{(2m)})$ and $X = (X', X'')$ we can write this action as multiplication from the left $gX'$ resp. $gX$.
\vspace{1em}


In order to determine the symplectic reduction we have to know the moment map on which we focus now:

\begin{prop}\label{SpImpuls}
The moment map of the action $Sp_{2n} \acts (\kxC^{2n})^{\oplus 2m}$ is
$$\mu: (\kxC^{2n})^{\oplus 2m} \rightarrow \liesp_{2n},\; X \mapsto \frac{1}{2}XQX^tJ.$$
\end{prop}

\begin{proof}
Written as matrices, the $Sp_{2n}$--equivariant pairing identifying $(\kxC^{2n})^{\oplus m}$ with its dual is $(\kxC^{2n})^{\oplus m} \times (\kxC^{2n})^{\oplus m} \to \kxC$, $(Y,Z) \mapsto \tr(Z^tJY)$,
so $\mu: (\kxC^{2n})^{\oplus 2m} \rightarrow \liesp_{2n}^*$, $\mu(X',X'')(A) = \tr((X'')^tJAX')$.
Using the fact $\tr((X'')^tJAX') = \tr(AX'(X'')^tJ)$ $= \tr(AX''(X')^tJ)$ we obtain $\mu(X',X'')(A) = \frac{1}{2}\tr(A(X'(X'')^t + X''(X')^t)J)$. As the argument of the trace is an element of $\liesp_{2n}$, identifying $\liesp_{2n}^*$ and $\liesp_{2n}$ via $\liesp_{2n}^* \longrightarrow \liesp_{2n}, \, (B \mapsto \tr(AB)) \longmapsto \frac{1}{2}(A + JA^tJ)$ leads to the modified moment map
\vspace{-1ex}
$$\mu: (\kxC^{2n})^{\oplus 2m} \rightarrow \liesp_{2n},\; (X',X'') \mapsto \frac{1}{2}(X'(X'')^t + X''(X')^t)J.$$
At last $XQX^t = (X',X'')\begin{pmatrix} 0 & I_m \\ I_m & 0\end{pmatrix} \begin{pmatrix} (X')^t \\ (X'')^t \end{pmatrix} = X'(X'')^t + X''(X')^t$.
\end{proof}

\begin{ex}
\label{exsl1}
As we want to consider the representations $\kxC^{m}$ of $Sl_2$ from Schwarz' table, we have to set $n = 1$ and $m = 1, 2, 3$. 
We are only interested in the ideal defined by the moment map rather than in the map itself, so we consider the slightly modified map $\widetilde{\mu}\colon (\kxC^2)^{\oplus 2m} \to \liesl_2$, $X \mapsto XQX^t$. 
This only changes the arrangement of the equations and avoids a cumbersome scalar factor. 

If $m = 1$, we have $XQX^t \hspace{-0.5ex}=\hspace{-0.5ex} \left(\begin{smallmatrix} 2x_{11}x_{12} & x_{11}x_{22}\pl x_{12}x_{21}\\x_{11}x_{22}\pl x_{12}x_{21} & 2x_{21}x_{22} \end{smallmatrix}\right)$.
Thus the momentum ideal is determined to be $\II_{\mu} = (x_{11}x_{12},\;\, x_{11}x_{22}+x_{12}x_{21},\;\, x_{21}x_{22})$.

Calculating the momentum ideal in the same way for $m = 2$ we obtain\hfill\par
$ \II_{\mu} = (x_{11}x_{13}\pl x_{12}x_{14},\;\, x_{11}x_{23}\pl x_{12}x_{24}+x_{13}x_{21}\pl x_{14}x_{22},\;\, x_{21}x_{23}\pl x_{22}x_{24}).$

If $m = 3$, the calculation yields that $\II_{\mu}$ is generated by $x_{11}x_{14}+x_{12}x_{15}+x_{13}x_{16}$, $x_{11}x_{24}+x_{12}x_{25}+x_{13}x_{26}+x_{14}x_{21}+x_{15}x_{22}+x_{16}x_{23}$ and $x_{21}x_{24}+x_{22}x_{25}+x_{23}x_{26}$.
\end{ex}


Before computing the quotient we examine if $\mu^{-1}(0)$ is normal according to proposition \ref{complint}.
We will prove that $\mu^{-1}(0)$ is a reduced resp. normal complete intersection if $m \geq 2n$ resp. $m \geq 2n + 1$; this is theorem \ref{thmnormal}. In the other cases the statement is not necessarily true: taking a look at example \ref{exsl1} where $m = 1$, we see $\det X =  x_{11}x_{22}-x_{12}x_{21} \not\in \II_{\mu}$, but $(\det X)^2 = (x_{11}x_{22}+x_{12}x_{21})^2- 4(x_{11}x_{12})(x_{21}x_{22}) \in \II_{\mu}$. Thus $\II_{\mu}$ is not a radical ideal and $\mu^{-1}(0)$ is not reduced, so it cannot be normal. Similarly $\mu^{-1}(0)$ is not normal if $n = 1$, $m = 2$. This time, $\mu^{-1}(0)$ is a reduced complete intersection according to theorem \ref{thmnormal}, but Jacobi's criterion shows that it is not regular in codimension $1$, so Serre's normality criterion fails. 
\vspace{1em}

To prove theorem \ref{thmnormal} we write $U$ for the symplectic vector space $\kxC^{2n}$ and $W$ for the euclidean vector space $\kxC^{2m}$, the euclidean structure being defined by $Q$. Then the moment map is
$$\mu\colon U \otimes W \to \liesp_{2n} \xrightarrow{\cong} \Sym_{2n}(\kxC),\,X \mapsto XQX^tJ \mapsto XQX^t.$$
The zero fibre $\mu^{-1}(0)$ is a complete intersection if and only if it is of maximal codimension, i.e. of dimension
$$\dim \mu^{-1}(0) = \dim U \otimes W - \dim \liesp_{2n} = 2n\cdot 2m - \binom{2n+1}{2} = 4mn-2n^2-n.$$
In addition $\mu^{-1}(0)$ is reduced if the codimension of its singular locus is at least $1$, and normal if it is at least $2$. Kaledin, Lehn and Sorger give a criterion on the isotropy group to check this:

\begin{lem}[\cite{kls:2006}]
Let $\mu\colon V \to \lieg^*$ be the moment map of an action $G \acts V$ and $N := \{x \in \mu^{-1}(0) \mid G_x \text{ not finite}\}$. Let $\ell := \dim V - \dim \lieg$ be the expected dimension of the zero fibre.
\begin{enumerate}
\item If $\dim N \leq \ell - 1$, then $\mu^{-1}(0)$ is a reduced complete intersection of dimension $\ell$,
\item if $\dim N \leq \ell - 2$, then $\mu^{-1}(0)$ is normal.
\end{enumerate}
\end{lem}

In the general case $U = \kxC^d$, $W = \kxC^k$ of arbitrary $d$ and $k$ and the above moment map we prove the following proposition. Together with the lemma this will give the desired result if we set $d = 2n$, $k = 2m$.

\begin{prop}
\label{spnormalreduziert}
For every $k \geq 2d$ we have
\begin{enumerate}
\item $\dim \mu^{-1}(0) = kd-\binom{d+1}{2} = kd - \frac{d^2}{2} - \frac{d}{2}$,
\item $\dim N \leq kd-\binom{d+1}{2}-1$.
If $k \geq 2d+1$ we even have $\dim N \leq kd-\binom{d+1}{2}-2$.
\end{enumerate}
\end{prop}

\begin{rem} In our case the definition of $Q$ requires $k$ to be even, but the proof works for arbitrary $k$ if $Q$ is replaced by any orthogonal structure. Similarly, as the moment map and the set $N$ are defined via the action $Sp(U) \acts U \otimes W$, a priori both only exist for even $d$. But the map $\mu\colon U \otimes W \to \Sym_d(\kxC)$, $X \mapsto XQX^t$ is also defined for odd $d$. 
Using $U\cong U^*$ we can write $U\otimes W \cong \Hom(U,W)$. Let $<u_1,\hdots,u_d>$ be a basis of $U$ and $<w_1,\hdots,w_k>$ an orthonormal basis of $W$ so that an element $X \in U\otimes W$ is mapped to some $\varphi\colon U \to W$ with corresponding matrix $X^t$ under this identification. Then
\begin{align*}
X \in \mu^{-1}(0) &\Longleftrightarrow XQX^t = 0 \Longleftrightarrow x_iQx_j^t = 0 \text{ for all rows } x_i,\, x_j \text{ of } X \\
&\Longleftrightarrow \varphi(u_i) \perp \varphi(u_j) = 0 \text{ for all basis vectors } u_i,\, u_j \text{ of } U \\
& \qquad \text{ w.r.t. the euclidean structure on } W \\
&\Longleftrightarrow \varphi(U) \subset W \text{ is an isotropic subspace of } W.
\end{align*}
This gives $\mu^{-1}(0) = \{\varphi \in \Hom(U,W) \mid \varphi(U) \subset W \text{ isotropic}\}$.

The set $N = \{X \in \mu^{-1}(0) \mid |Sp(U)_X| = \infty\}$  of elements with infinite isotropy group is contained in the set $M:= \{X \in \mu^{-1}(0) \mid Sp(U)_X \neq 1\}$ of elements with non-trivial isotropy group, so it suffices to restrict the dimension of $M$.\par
Let $g \in Sp(U)$. As $Sp(U)$ acts trivially on $W$ we have $\Hom(U,W)^g = \Hom(U^g,W)$. Thus $\varphi$ is invariant under $g$ if and only if $\varphi$ remains fixed under composition with the map $g\colon U \to U$. This leads to the following criterion for the isotropy group of $\varphi$ to be non-trivial:
\begin{align*}
Sp(U)_{\varphi} \neq 1 &\Longleftrightarrow \exists\; 1 \neq g \in Sp(U): \varphi \circ g = \varphi \nonumber\\
&\Longleftrightarrow \exists\;1 \neq g \in Sp(U):\;\forall \; u \in U: \varphi(g(u)) = \varphi(u).
\end{align*}
Let $g$ be fixed with this property and choose $u \in U$ such that $g(u) \neq u$. Then we have $\varphi(u \mi g(u)) = \varphi(u) \mi \varphi(g(u)) = 0$, i.e. $0 \neq u \mi g(u) =: u_0 \in \ker \varphi$. In particular the kernel contains a line $\kxC u_0$ and $\varphi$ factorises via $U/\kxC u_0$. This means $N \subset M \subset L$ where $L:= \{\varphi \in \mu^{-1}(0) \mid \ker \varphi \neq 0\} = \{\varphi \in \mu^{-1}(0) \mid \exists\; u_0: \varphi\colon U/\kxC u_0 \to W\}$, and it suffices to restrict the dimension of $L$, which is independent of the structure of the symplectic group. Therefore $L$ also exists for odd $d$.
\end{rem}

\begin{proof}
Let $d = 1$. Then by assumption $k \geq 2$. We have to show $\dim \mu^{-1}(0) = k - 1$, $\dim L \leq k - 2$ for every $k$ and $\dim L \leq k - 3$ if $k \geq 3$, shortly $\dim L = 0$.\par
Let $U = <u>$ and $<w_1,\hdots,w_k>$ an orthonormal basis of $W$. Let $\varphi$ be defined by $u \mapsto \sum_{i=1}^k a_iw_i$. Then $\varphi(U)$ is isotropic if and only if $\sum_{i=1}^k a_i^2 = 0$. Thus we can choose $k-1$ coefficients independently, the last one is fixed up to sign. In particular $\dim \mu^{-1}(0) = k - 1$. 
If the kernel of $\varphi$ is non-trivial we have $\ker \varphi = U$ and $\varphi\colon \{0\} \to W$. Thus $L$ has dimension $0$.\par
Now consider a $d$--dimensional space $U$ together with its isotropic images $\varphi(U)$ of dimension $d$. Every $\varphi \in \mu^{-1}(0)$ maps to such a subspace; if the image is not of dimension $d$, complete it to a $d$--dimensional isotropic space. Thus $\varphi(U)$ varies in the Grassmannian of $d$--dimensional isotropic subspaces in the euclidean space $W$ of dimension $k$. The set of such subspaces has dimension $\dim (\Grass_{iso}(d,W)) = d(k - \frac{3}{2}d-\frac{1}{2})$. There are $d^2$ maps $\varphi \in \Hom(U,W)$ to every isotropic subspace of dimension $d$. Thus the zero fibre of the moment map is at most of dimension
\begin{align*}
\dim \mu^{-1}(0) &\leq d(k-\frac{3}{2}d-\frac{1}{2}) + d^2 =  dk - \binom{d+1}{2}.
\end{align*}
On the other hand, as every equation reduces the dimension at most by one according to Krull's theorem and since being isotropic is described by $\binom{d+1}{2}$ equations for a $d$--dimensional euclidean space, we obtain equality.

Now if $\varphi \in L$ factorises via $U/\kxC u_0$ for some $u_0 \in \ker \varphi$, then if $\varphi(U)$ is isotropic so is $\varphi(U/\kxC u_0)$. Thus $\varphi$ is contained in $\mu_0^{-1}(0)$ where $\mu_0\colon (U/\kxC u_0)\otimes W \to \Sym_{d-1}(\kxC)$, $\overline{X} \mapsto \overline{X}Q\overline{X}^t$ denotes the moment map in smaller dimension. Using the induction hypothesis we know $\dim(\mu_0^{-1}(0)) = (d \mi 1)k - \binom{d}{2}.$
There are $\dim \prP(U) = d - 1$ different subspaces $\kxC u_0$ of $U$. This yields
\begin{align*}
\dim L &\leq (d-1)k - \binom{d}{2} + (d-1) = dk - \binom{d+1}{2}-(k-2d+1)\\
&\leq\left\{\begin{array}{ll} dk - \binom{d+1}{2}-1 & \text{if } k \geq 2d, \\ dk - \binom{d+1}{2}-2 & \text{if } k \geq 2d + 1.\end{array}\right. \qedhere
\end{align*}
\vspace{-1em}
\end{proof}

Setting $d = 2n$, $k = 2m$ we deduce normality in the case $m \geq 2n + 1$. This completes the proof of theorem \ref{thmnormal}.
\vspace{1ex}

Looking at Schwarz' list we see that we have to reduce ourselves to the case $m \leq 2n + 1$. Let us first compute $(\kxC^{2n})^{\oplus 2m}\red Sp_{2n}$ by describing the invariants and relations. Then we turn to the symplectic reduction $\mu^{-1}(0)\red Sp_{2n}$. By the preceding theorem we can only guarantee symplecticity in the case $m = 2n + 1$.
\vspace{1em}

According to the first fundamental theorem for $Sp_{2n}$ (cf. \cite{wey:1946}) all invariants of our action are
\vspace{-0.5em}
$$ z_{ij} := (x^{(i)})^tJx^{(j)},\quad i < j, \quad i,j \in \{ 1, \hdots, s \}, $$
where $s = m$ for the simple action and $s = 2m$ for the duplicated one.
As matrices the invariants can be arranged as $(X')^tJX'$ resp. $X^tJX$, where you find the $z_{ij}$ above the diagonal and their negatives below, the diagonal is zero. We will often use the permutation $X^tJXQ$.
By the second fundamental theorem the $\binom{s}{2}$ invariants do not have any relations for $s \leq 2n + 1$, these are exactly the cases in Schwarz' table. For larger $s$ there are the relations $J_k = 0, \, k = 1, \hdots, n+1$, where
$J_k = \sum_{\pi \in S_{2n+1}} \sgn \pi \; z_{\pi(i_0)j_0} \cdots z_{\pi(i_{2k-2})j_{2k-2}} z_{\pi(i_{2k-1})\pi(i_{2k})} \cdots z_{\pi(i_{2n-1})\pi(i_{2n})}$ 
for  $i_k$ and $j_l \in \{ 1, \hdots, 2m \}$ pairwise disjoint.

In our example $n = 1$, the invariants specialise to $z_{ij} = x_{1i}x_{2j} - x_{1j}x_{2i}$, $i<j$. These are exactly the $2 \times 2$--minors of $X$.
So $\varrho_1^1$ does not have any invariant, $\varrho_2^1$ has exactly one, namely the determinant of $X$, and there are three invariants for $\varrho_3^1$. Of course, these invariants of coregular representations do not fulfill any relation.
Unlike this, the six invariants of $\varrho_4^1$ are connected by one Plücker relation $J_1 = z_{12}z_{34} - z_{13}z_{24} + z_{14}z_{23}$, and the fifteen invariants of $\varrho_6^1$ satisfy fifteen of such Plücker relations, namely $z_{ij}z_{kl} - z_{ik}z_{jl} + z_{il}z_{jk}$, $1 \leq i < j < k < l \leq 6.$
In this last case there is one further relation \par
$J_2 = z_{14}z_{25}z_{36} + z_{24}z_{35}z_{16} + z_{34}z_{15}z_{26} - z_{14}z_{35}z_{26} - z_{24}z_{15}z_{36} - z_{34}z_{25}z_{16}.$
\vspace{1ex}

Now let us compute $\mu^{-1}(0)\red Sp_{2n}$ for arbritrary $m$ and $n$:

\begin{prop}
\label{SpQuot}
As a set the quotient of the action $Sp_{2n} \acts (\kxC^{2n})^{\oplus 2m}$ is
$$ \mu^{-1}(0) \red Sp_{2n} \cong Z := \{ A \in \lieso_{2m} \mid A^2 = 0,\,\rk A \leq \min\{2n,m\} \}. $$
\end{prop}
\begin{proof} 
By proposition \ref{SpImpuls} the preimage of zero under the moment map is
\begin{equation*} \mu^{-1}(0) = \{ X \in (\kxC^{2n})^{\oplus 2m} \mid XQX^tJ = 0\}. \end{equation*}
Define the map $\nu: \mu^{-1}(0) \to \lieso_{2m}$, $X \mapsto X^tJXQ$. At first we show that the image of $\nu$ is contained in $Z$. For this choose an arbitrary $X \in \mu^{-1}(0)$:
\begin{itemize}
\item $-Q(X^tJXQ)^tQ = -QQ^{-1}X^t(-J)XQ = X^tJXQ$, i.e. $X^tJXQ \in \lieso_{2m}$,
\item $(X^tJXQ)^2 = X^tJ(XQX^tJ)XQ = 0$,
\item $\rk(X^tJXQ) < m$ holds because $(X^tJXQ)^2 = 0$. The other bound arises from $\rk(X^tJXQ) \leq \min\{\rk X, \rk J, \rk Q\} \leq \min\{2n,2m\} \leq 2n$.
\end{itemize}
To see that $\nu$ factorises via $\mu^{-1}(0) \red Sp_{2n}$ we have to show that $\nu$ is constant on each orbit. Indeed for every $g \in Sp_{2n}$: 
$\nu(gX) = (gX)^tJgXQ  = X^tJXQ = \nu(X)$.

For the injectivity of the induced map $\overline{\nu}$ on the quotient we look at the following  diagram describing the correspondence between algebras and varieties. One obtains the second diagram from the first one by taking the spectrum of each ring and by reversing all the arrows. Furthermore we utilise $\im \nu \subset Z$. 
We write $V := (\kxC^{2n})^{\oplus 2m}$ for short.
\begin{equation*} 
\begin{xy} 
\xymatrix{
\kxC[V] \ar@/_9mm/@{<-}[dd] \ar@{<-_)}[d] \ar@{->>}[r] & \kxC[V]/\II_{\mu} \ar@{<-_)}[d] \\
\kxC[V]^{Sp_{2n}} \ar@{<<-}[d]^{\varphi} \ar@{->>}[r] & (\kxC[V]/\II_{\mu})^{Sp_{2n}}\\
\kxC[z_{ij}] \ar@{->>}^{\beta}[ur] &  
}\end{xy}
\qquad
\begin{xy} 
\xymatrix{
V \ar@/_9mm/@{->}[dd]_{\psi} \ar@{->>}[d] \ar@{<-^)}[r] & \mu^{-1}(0) \ar@{->>}[d] \ar@/^11mm/@{->}[dd]^{\nu} \\
V \red Sp_{2n} \ar@{_(->}[d]_{\overline{\psi}} \ar@{<-^)}[r] & \mu^{-1}(0) \red Sp_{2n} \ar@{_(->}[d]_{\overline{\nu}}\\
\kxC^{\binom{2m}{2}} = \lieso_{2m} \ar@{<-^)}[r] & Z }
\end{xy} 
\end{equation*}
As map of invariants $\varphi:\kxC[z_{ij}] \to \kxC[V]^{Sp_{2n}},\, z_{ij} \mapsto (x^{(i)})^tJx^{(j)}$, is surjective, for which reason the corresponding map of varieties $\overline{\psi}:V\red Sp_{2n} \to \lieso_{2m}$ is injective. Here $\overline{\psi}$ is induced by $\psi:V \to \lieso_{2m},\, X \mapsto X^tJXQ$, and we have $\psi|_{\mu^{-1}(0)} = \nu$. As $\overline{\psi}$ is injective so is $\overline{\nu} = \overline{\psi}|_{\mu^{-1}(0)\red Sp_{2n}}: \mu^{-1}(0)\red Sp_{2n} \to Z$. \par

For showing surjectivity choose $A \in \lieso_{2m}$ satisfying $A^2 = 0$ and $\rk A \leq \min\{2n,m\}$. We will construct an $X \in  \mu^{-1}(0)$ such that $A = X^tJXQ$.

As $AQ$ is skew-symmetric there exists $S \in Gl_{2m}$ such that $AQ$ obtains the shape
$$ (S ^{-1})^tAQS ^{-1} = \begin{pmatrix} \widetilde A & 0 \\ 0 & 0 \end{pmatrix}, \text{ i.e. } A = S^t \begin{pmatrix} \widetilde A & 0 \\ 0 & 0 \end{pmatrix} SQ,$$
where $\widetilde A$ is regular of size $\ell := \rk A$ and skew-symmetric. But the size of a non-degenerate skew-symmetric matrix is even, so $\ell = 2k$, and further there exists an $R \in Gl_{2k}$ such that $\widetilde{A} = R^tJ_{2k}R.$
Inserting this in the above expressions and writing $\widetilde R := (R,\,0) \in \Mat_{2k \times 2m}$ and $T := \widetilde R S$, we obtain
$$AQ = S^t \begin{pmatrix} R^tJ_{2k}R & 0 \\ 0 & 0 \end{pmatrix} S = S^t \widetilde{R}^t J_{2k} \widetilde R S = T^tJ_{2k}T.$$
If $k = n$, then $X := T$ is the desired preimage of $A$. Otherwise $k < n$ and we define $U \in \Mat_{2n \times 2k}$ via
$U := \frac{1}{\sqrt{2}}\begin{pmatrix} \binom{I_k}{0} & \binom{-I_k}{0} \\ \binom{I_k}{0} & \binom{I_k}{0} \end{pmatrix}$. One easily computes $U^tJ_{2n}U = J_{2k}$. 
Setting $X := UT$ finally yields $AQ = T^tJ_{2k}T = T^tU^tJ_{2n}UT = X^tJ_{2n}X$.

It remains to prove $X \in \mu^{-1}(0)$. As  $A^2 = 0$ we have $X^tJXQX^tJXQ = 0$. By construction $X = (R,\,0)S$ resp. $X = U(R,\,0)S$ is regular on rows, thus so is $XQ$, and $X^tJ$ is regular on columns. This shows that if $XQX^tJ$ was different from zero then $XQX^tJ \cdot XQ$ and $X^tJ \cdot XQX^tJXQ$ would also be, a contradiction. Therefore $XQX^tJ = 0$ must hold and $X$ is contained in $\mu^{-1}(0)$.
\end{proof}


Until now we know the quotient only as a set. Our next aim is to describe its geometric structure.\par
If $\mu^{-1}(0)$ is reduced, the corresponding ideal is its own radical, but the defining equations for $Z$ do not yield a radical ideal in general. So for $\mu^{-1}(0)\red Sp_{2n}$ and $Z$ to be equal as varieties we have to equip $Z$ with its reduced structure, which will be assumed in the following. In the examples $n = 1$, $m = 2$ or $3$ the reduced structure is attained by adding some equations involving the Pfaffian as we will see later.
\vspace{1ex}

On $(\kxC^{2n})^{\oplus 2m} = \kxC^{2n} \otimes \kxC^{2m}$ there is not only the $Sp_{2n}$--action from the left, but also an $SO_{2m}$--action, namely multiplication from the right. This action leaves $\mu^{-1}(0)$ invariant, for $(Xg)Q(Xg)^tJ =  X(gQg^t)X^tJ = XQX^tJ$ if $g \in SO_{2m}$. Thus $SO_{2m}$ also acts on $\mu^{-1}(0)$, and on the quotient $\mu^{-1}(0)\red Sp_{2n}$ because the $Sp_{2n}$--action is independent of the $SO_{2m}$--action. 
With one element its whole orbit is contained in $\mu^{-1}(0)\red Sp_{2n}$ by $SO$--invariance. Therefore $Z$ is the union of certain nilpotent orbits of $\lieso_{2m}$, i.e. of orbits of nilpotent elements under the adjoint representation.\par
Let $d_1 \geq \hdots \geq d_k$ be a partition of $2m$. All but one partitions whose even parts occur with even multiplicity correspond to exactly one nilpotent orbit $\mathcal{O}_{[d_1,\hdots,d_k]}$ of $\lieso_{2m}$. Only the very even partitions, which have only even parts $d_i$ with even multiplicity, correspond to two nilpotent orbits $\mathcal{O}^{I}_{[d_1,\hdots,d_k]}$ and $\mathcal{O}^{II}_{[d_1,\hdots,d_k]}$ (cf. \cite{cmg:1993}).  The Jordan form of a representative of this orbit has blocks of size $d_1,\,\hdots,\,d_k$.

\begin{prop}
\label{nilOrb}
The variety $Z$ consists of the following nilpotent orbits:
$$ Z = \left\{ \begin{array}{ll} 
\bigcup\limits_{k=0}^{\frac{m-1}{2}}  \mathcal{O}_{[2^{2k},1^{2(m-2k)}]} = \overline{\mathcal{O}}_{[2^{m-1},1^2]}, & \text{ if } m \text{ is odd,} \vspace{0.5ex}\\
\bigcup\limits_{k=0}^{\frac{m}{2}-1}  \mathcal{O}_{[2^{2k},1^{2(m-2k)}]} \cup \mathcal{O}^{I}_{[2^{m}]} \cup \mathcal{O}^{II}_{[2^{m}]} = \overline{\mathcal{O}}^{I}_{[2^{m}]} \cup \overline{\mathcal{O}}^{II}_{[2^{m}]}, & \text{ if } m \text{ is even.}
\end{array} \right.$$
\end{prop}
\begin{proof} Obviously the union of the indicated orbits is the closure of the orbits with highest rank in both cases.\par
"$\subseteq$": Because of $A^2 = 0$ each element $A \in Z$ can have Jordan blocks of size at most $2$. The rank condition on $Z$ implies that the Jordan form of an element in $Z$ has at most $m$ blocks of size $2$ if $m \leq 2n$ and at most $2n$ such blocks if $m = 2n + 1$, since every $2\times 2$--block raises the rank exactly by $1$. The indicated unions contain all Jordan forms of these types, therefore they contain $Z$.\par
"$\supseteq$": Conversely let $X$ be a matrix in one of the indicated orbits. Then $X^2 = 0$ because the Jordan form of $X$ consists of blocks of size at most $2$. The maximal number of such blocks implies $\rk X \leq \min\{2n,m\}$ in both cases. Thus $X \in Z$.\par
The closure of a nilpotent orbit is always reduced. So the orbit closures and $Z$ also agree as varieties, both describe $(\kxC^{2n})^{\oplus 2m}\sred Sp_{2n}$ if and only if $\mu^{-1}(0)$ is reduced.
\end{proof}

To analyse normality of nilpotent orbits we work with \cite{kp:1982}. From the criterion for normality stated there together with table 3.4, section 2.3 and the theorem in section 17.3 of loc. cit. we deduce:

\begin{prop}\label{znorm}
$Z$ is normal if and only if $m$ is odd. If $m$ is even, $Z$ decomposes into two normal components. The same holds for $(\kxC^{2n})^{\oplus 2m}\sred Sp_{2n}$ if $\mu^{-1}(0)$ is reduced.
\end{prop}

Fu gives a criterion for nilpotent orbit closures to be symplectically resolvable:

\begin{prop}[Fu, \cite{fu:2003}]
Let $\mathcal{O}$ be a nilpotent orbit of $\lieso_{2m}$ associated to the partition $d = [d_1,\hdots,d_N]$. Its closure $\overline{\mathcal{O}}$ has a symplectic resolution if and only if there is an even $q \neq 2$ such that $d_1,\hdots,d_q$ are odd and $d_{q+1},\hdots,d_N$ are even, or if there are exactly two odd parts at the positions $2k-1$ and $2k$ for some $k \in \mathbbm{N}$.
\end{prop}

If $m$ is even, both orbit closures $\overline{\mathcal{O}}^I_{[2^m]}$ and $\overline{\mathcal{O}}^{II}_{[2^m]}$ admit a symplectic resolution: set $q = 0$ in the proposition. The union of these constitutes a symplectic resolution for $Z$.
For orbits associated to the partitions $[2^{m-1},1^2]$ the first condition of the proposition never holds. The second condition is fulfilled because there occur exactly two ones at the  positions $2m\mi 1$ and $2m$. So $Z$ also admits a symplectic resolution for odd $m$.
In particular, $(\kxC^{2n})^{4n+2}\sred Sp_{2n}$ has a symplectic resolution.

\begin{rem}
As $\mu^{-1}(0)$ and hence $\mu^{-1}(0)\red Sp_{2n}$ is normal, the existence of a symplectic resolution implies that $(\kxC^{2n})^{4n+2}\sred Sp_{2n}$ is a symplectic variety indeed.
\end{rem}

Now we take a look at the quotients $\mu^{-1}(0)\red Sl_2$. First we consider the double of the case $m = 1$. Up to multiples $z^2 = 0$ is the only relation of the single invariant $z := \det X$, so the quotient is the non--reduced variety
$$\kxC^2 \oplus \kxC^2 \sred Sl_2 = \Spec\bigl(\kxC[z]/(z^2)\bigr).$$
This is not a symplectic variety.

\pagebreak

For the double of $\varrho_2^1$, the six invariants satisfy eleven relations modulo $\II_{\mu}$.
Identifying the invariants with the entries of a matrix $A \in \lieso_4$ yields an isomorphism of varieties
\vspace{-0.5em}
$$ (\kxC^2 \oplus \kxC^2)^{\oplus 2} \sred Sl_2 \cong \{ A \in \lieso_4 \mid A^2 = 0,\, \Pf(QA) = 0\},$$
where the Pfaffian guarantees reducedness on the right. The quotient is not normal, but the union of two normal components by proposition \ref{znorm}. 
On the contrary, according to the general case, the preimage of the moment map is normal if $m = 3$. The quotient is the six dimensional symplectic variety and nilpotent orbit closure
$$ (\kxC^2)^{\oplus 6} \sred Sl_2 \cong \{ A \in \lieso_6 \mid A^2 = 0,\, \rk A \leq 2, \, \Pf_4(QA) = 0 \} = \overline{\mathcal{O}}_{[2^2,1^2]},$$
where $\Pf_4(QA)$ denotes the Pfaffians of the $15$ skew-symmetric $4 \times 4$--minors of $QA$, which again assure reducedness.
\vspace{1ex}

It is a well-known fact that the singular locus of the nilpotent orbit $\mathcal{O}_{[2^2,1^2]}$ is the nilpotent orbit $\mathcal{O}_{[1^6]}$, which is nothing else but the origin.
According to \cite{fu:2003} the nilpotent orbit closure $\overline{\mathcal{O}}_{[2^2,1^2]}$ has a symplectic resolution. 
The isomorphism of Lie algebras $\lieso_6 \cong \liesl_4$ identifies $\{ A \in \lieso_6 \mid A^2 = 0,\, \rk A \leq 2, \, \Pf_4(QA) = 0 \}$ with  $Y := \{B \in \liesl_4 \mid \rk B \leq 1\}$. This variety, and hence $(\kxC^2)^{\oplus 6} \sred Sl_2$, has two well-known symplectic resolutions by the cotangent bundle and its dual:
$$\begin{array}{l}
\{(A,L) \in Y \times \prP^3 \mid \im A \subset L\} \to Y,\\
\{(A,H) \in Y \times (\prP^3)^* \mid H \subset \ker A\} \to Y.
\end{array}$$

\begin{rem}
Note that the singular locus is of codimension $6$. Since there is a symplectic resolution, the singularity cannot be locally $\mathbbm{Q}$-factorial by \cite{nam:2006}.
\end{rem}

\section{The action of $Sl_2$ on symmetric powers and on its Lie algebra}

Next we present the classical cases of the action of $Sl_2$ on the symmetric powers $S^3\kxC^2$ and $S^4\kxC^2$ whose invariants and relations have already been determined by Hilbert in \cite{hil:1890, hil:1893}. The symplectic reduction of both actions is isomorphic to a quotient of a finite group action. This was expected since both are so-called polar representations, see \cite{cle:2007}.

Afterwards we overview the action of $Sl_2$ on one and two copies of its Lie algebra. We use the known invariants to compute the symplectic reduction.

This section shows that the quotients of both types (2) and (3) of representations of $Sl_2$ are symplectic varieties and admit a symplectic resolution.

\subsection{The action of $Sl_2$ on symmetric powers}\hfill\par

We consider elements of $S^3\kxC^2$ as binary cubics $A := a_0x^3 + 3a_1x^2y + 3a_2xy^2 + a_3y^3$ and elements of $S^4\kxC^2$ as polynomials $A := a_0x^4 + 4a_1x^3y + 6a_2x^2y^2 + 4a_3xy^3 + a_4y^4$. 
The actions of $Sl_2$ on $S^3\kxC^2$ and $S^4\kxC^2$ are induced by the action of $Sl_2$ on $\kxC^2$, where the action is multiplication from the left, via $g \cdot (x^{i}y^{j}) = (gx)^{i}(gy)^{j}$.

Both symmetric powers have a non-degenerate $Sl_2$--invariant pairing, given by
$$\begin{array}{l}
\sigma\phantom{'}\colon S^3\kxC^2 \times S^3\kxC^2 \to \kxC,\; (A, B) \mapsto a_0b_3-a_3b_0-3(a_1b_2-a_2b_1),\\
\sigma'\colon S^4\kxC^2 \times S^4\kxC^2 \to \kxC,\; (A,B) \mapsto a_0 b_4 + b_0 a_4 - 4 a_1 b_3 - 4 b_1 a_3 + 6 a_2 b_2.
\end{array}$$
So both actions in consideration are self-dual.
\vspace{1em}

There is only one invariant for the simple action of $Sl_2$ on $S^3\kxC^2$, namely the discriminant $d = - 4a_0a_2^3 - 4a_1^3a_3 - a_0^2a_3^2 + 3a_1^2a_2^2 + 6a_0a_1a_2a_3$. 
On the double there are seven invariants: An invariant of degree $2$ is $F := \sigma(A,B)$, five invariants of degree $4$ are obtained by polarising the discrimanant and an invariant of degree $6$ is the resultant. These invariants fulfill two relations of degree $8$ and $10$.

The moment map for this action is $ \mu\colon S^3\kxC^2 \oplus S^3\kxC^2 \longrightarrow \liesl_2^*$, where
$\mu(A,B)(\xi) = 3\big((a_0\xi_{11} + a_1\xi_{12})b_3 - (a_0\xi_{21} + 3a_1\xi_{11} + 2a_2\xi_{12})b_2 +(2a_1\xi_{21} - a_2\xi_{11} + a_3\xi_{12})b_1 - (a_2\xi_{21} - a_3\xi_{11})b_0\big)$.
Generators of the corresponding momentum ideal $\II_{\mu}$ are given by $a_1b_3+a_3b_1-2a_2b_2$, $a_0b_3+a_3b_0-a_1b_2-a_2b_1$ and $2a_1b_1-a_0b_2-a_2b_0$.

Modulo $\II_{\mu}$ there remain three invariants $F = \sigma(A,B)$, $d_0 = \disc(A)$, $d_4 = \disc(B)$ with one relation $16d_0d_4-F^4$. Up to scaling this is the equation of an $A_3$--singularity in $\{0\}$. Therefore $S^3\kxC^2 \oplus S^3\kxC^2 \sred Sl_2 \cong \kxC^2/ (\intZ/4)$ is a symplectic variety and symplectically resolvable by two successive blow ups in the origin.
\vspace{1em}

The simple action of $Sl_2$ on $S^4\kxC^2$ has two invariants, which are the quadratic form 
$Q = a_0 a_4 - 4 a_1 a_3 + 3 a_2^2$ and the Catalectian 
$C = \det \left( \begin{smallmatrix} a_0 & a_1 & a_2 \\ a_1 & a_2 & a_3 \\ a_2 & a_3 & a_4 \end{smallmatrix}\right)$. 
Polarising these yields three quadratic invariants $Q_0, Q_1, Q_2$ and four cubic invariants $C_0, C_1, C_2, C_3$, which are algebraically independent. An additional invariant is 
$T := \det\left(\begin{smallmatrix} a_0 & 3a_1 & 3a_2 & a_3 \\ a_1 & 3a_2 & 3a_3 & a_4 \\ b_0 & 3b_1 & 3b_2 & b_3 \\ b_1 & 3b_2 & 3b_3 & b_4 \end{smallmatrix}\right)$.
These eight invariants generate the ring of invariants. There is one relation in degree 12.


Here the moment map is $\mu\colon S^4\kxC^2 \oplus S^4\kxC^2 \to \liesl_2^*$, 
$\mu(A,B)(\xi) = 4\bigl( (\xi_{12}a_1 \plu a_0\xi_{11})b_4 -(\xi_{21}a_0 + 2a_1\xi_{11} + 3\xi_{12}a_2)b_3 + 3(\xi_{12}a_3 - \xi_{21}a_1)b_2 - (3\xi_{21}a_2 + \xi_{12}a_4 - 2a_3\xi_{11})b_1 + (\xi_{21}a_3 - a_4\xi_{11})b_0\bigr)$,
with momentum ideal $\II_{\mu}$ generated by $b_4 a_1 \mi 3 b_3 a_2 \mi b_1 a_4 \pl 3 b_2 a_3$, $a_0 b_4 \mi b_0 a_4 \mi 2 a_1 b_3 \pl 2 b_1 a_3$ and $b_0 a_3 \mi b_3 a_0 \mi 3 b_1 a_2 \pl 3 b_2 a_1$.

The ring of coordinates of the quotient $\mu^{-1}(0)\red Sl_2$ is generated by $Q_0$, $Q_1$, $Q_2$, $C_0$, $C_1$, $C_2$, $C_3$.
Comparing these invariants and their relations with the invariants and relations of the action of the symmetric group $S_3$ on the double of $(\kxC^3)_0 = \{ (y_1\;y_2\;y_3)^t \in \kxC^3 \mid y_1 + y_2 + y_3 = 0 \} \cong \kxC^2$ via permutation of indices, we obtain the same quotient. Thus it is a symplectic variety as a subset of $S^3\kxC^2 = (\kxC^2)^{\oplus 3}/S_3$.

The singular locus of $(\kxC^3)_0 \oplus (\kxC^3)_0 / S_3$ can be resolved by a Hilbert scheme:  
The barycentral map $s:(\kxC^2)^{\oplus 3} \to \kxC^2$, $\bigl( \binom{x_1}{y_1},\binom{x_2}{y_2},\binom{x_3}{y_3} \bigr) \mapsto \frac{1}{3}\binom{x_1+x_2+x_3}{y_1+y_2+y_3}$
is $S_3$--invariant and therefore factorises via $(\kxC^2)^{\oplus 3}/S_3$. The preimage of $0$ under the induced map $\overline{s}$ is exactly our quotient $(\kxC^3)_0 \oplus (\kxC^3)_0 / S_3$. Composing $\overline s$ with the Hilbert--Chow--morphism $\rho\colon \Hilb^3(\kxC^2) \to S^3\kxC^2$, which is a resolution of $S^3\kxC^2$ due to Fogarty \cite{fog:1968}, and even a symplectic one, cf. \cite{bea:1983}, we obtain
$$ \begin{array}{ccccc}
\Hilb^3(\kxC^2) & \xrightarrow{\quad \rho\quad} & S^3\kxC^2 & \xrightarrow{\quad \overline{s}\quad} & \kxC^2 \\
\vertsub & & \vertsub & & \vertin \\
\rho^{-1}((\kxC^3)_0 \oplus (\kxC^3)_0/S_3) & \xrightarrow{\quad \phantom{\rho}\quad} & (\kxC^3)_0 \oplus (\kxC^3)_0/S_3 & \xrightarrow{\quad \phantom{\overline{s}}\quad} & 0. \end{array}$$
So the restriction of the Hilbert--Chow--morphism to $\rho^{-1}((\kxC^3)_0 \oplus (\kxC^3)_0/S_3)$ is a symplectic resolution for $(\kxC^3)_0 \oplus (\kxC^3)_0/S_3 = S^4\kxC^2 \oplus S^4\kxC^2 \sred Sl_2$.

\subsection{The action of $Sl_2$ on its Lie algebra}\hfill\par
\label{abslie}

The action of $Sl_2$ on $\liesl_2$ is just the adjoint representation $\Ad(g)(A) = gAg^{-1}$. 
The map $\Ad$ maps $Sl_2$ to $Gl(\liesl_2)$, even to the orthogonal group $SO_3$. This is the well-known $2:1$--covering.  
Identifying $\liesl_2$ with $\kxC^3$, the actions $Sl_2 \acts \liesl_2$ and $SO_3 \acts \kxC^3$ coincide via this covering, so they have the same rings of invariants, even at several copies:
$\kxC[(\kxC^3)^{\oplus k}]^{SO_3} = \kxC[\liesl_2^{\oplus k}]^{Sl_2}$.
So instead of analysing $Sl_2 \acts \liesl_2$, we consider the action of $SO_3$ on $\kxC^3$, where we dispose of the fundamental theorems, to compute the symplectic reductions $\liesl_2^{\oplus 2} \sred Sl_2$ and $\liesl_2^{\oplus 4} \sred Sl_2$.
\vspace{1ex}

The action we have to consider now is
$$\vartheta_n \colon SO_3 \times (\kxC^3)^{\oplus n} \to (\kxC^3)^{\oplus n},\; (g, x^{(1)}, \hdots, x^{(n)}) \mapsto (gx^{(1)}, \hdots, gx^{(n)})$$
in the cases $n = 1$, $n = 2$, resp. $n = 2$, $n = 4$ for the doubles because the action is self-dual since $(x,y) \mapsto y^tQx$ gives an invariant non-degenerate pairing on $\kxC^3$.
Writing $X' = (x^{(1)} \mid \hdots \mid x^{(n)})$, $X'' = (x^{(n+1)} \mid \hdots \mid x^{(2n)})$ and $X = (X', X'')$, the pairing identifying $(\kxC^{3})^{\oplus n}$ with its dual takes the shape
$(\kxC^3)^{\oplus n} \times (\kxC^3)^{\oplus n} \to \kxC$, $(X',X'') \mapsto \tr((X'')^tQX')$.

According to the first fundamental theorem for $SO_3$ all invariants are
$$\begin{array}{ll}
t_{ij} := (x^{(i)})^tQx^{(j)}, &i \leq j,\quad i,j \in \{ 1, \hdots, n \}, \\
\det(x^{(i_1)} \mid x^{(i_2)} \mid x^{(i_3)}), &i_1 < i_2 < i_3, \quad i_1, i_2 ,i_3 \in \{1,\hdots,n\}.
\end{array}$$
In the cases $n = 1$, $n = 2$ there are less than three vectors, so the second type of invariants does not appear. The only invariant of $\vartheta_1$ is $x^tQx$, 
$\vartheta_2$ has three invariants $x^tQx$, $x^tQy$ and $y^tQy$ if we set $x := x^{(1)}$, $y := x^{(2)}$. Both actions do not have any relations by the second fundamental theorem, hence they are in Schwarz' list (table $3a$, item $1$ and $4$).

For $n = 4$ there are ten invariants of the first type and four invariants of the second type. For a clearly arranged description of the relations we refer to \cite{ls:2006}.

With analogous calculations to the ones of proposition \ref{SpImpuls} we obtain the moment map
$\mu\colon (\kxC^m)^{\oplus 2n} \rightarrow \lieso_m,\; (X',X'') \mapsto \frac{1}{2}(X'(X'')^t - X''(X')^t)Q = -\frac{1}{2}XJX^tQ.$
In the special case of $\vartheta_1$ doubled the entries of $\mu(x,y)$ provide the momentum ideal
$\II_{\mu} = \left( x_1y_2 - x_2y_1,\; x_1y_3 - x_3y_1,\; x_2y_3 - x_3y_2 \right)$.
Writing $X' = (x \mid y)$, $X'' = (z \mid u)$ in case of the double of $\vartheta_2$, the ideal
$\II_{\mu}$ is generated by $x_1z_2 + y_1u_2 - x_2z_1 - y_2u_1$, $x_1z_3 + y_1u_3 - x_3z_1 - y_3u_1$ and $x_2z_3 + y_2u_3 - x_3z_2 - y_3u_2.$
\vspace{1ex}

Now we reduce the invariants and relations modulo the momentum ideal. 
The invariants $t_{11} = x^tQx$, $t_{12} = x^tQy$, $t_{22} = y^tQy$ of the double of $\vartheta_1$ satisfy a single relation modulo $\II_{\mu}$, namely $t_{12}^2 - t_{11}t_{22} = 0$. Thus $\mu^{-1}(0) \red SO_3$, and with it $\liesl_2^{\oplus 2} \sred Sl_2$, has only an isolated $A_1$--singularity. This shows that the quotient has a symplectic structure. The $A_1$--singularity admits a symplectic resolution by blowing up the origin.

For the second example it turns out that the determinantal invariants $T_1,\hdots,T_4$ are dispensable modulo $\II_{\mu}$, the other ten invariants fulfill six relations. The quotient $\mu^{-1}(0) \red SO_3 = (\liesl_2)^{\oplus 4} \sred Sl_2$ is isomorphic to the nilpotent orbit closure $\overline{\mathcal{O}}_{[2^2]} = \{A \in \liesp_4 \mid A^2 = 0\}$ and hence is a symplectic variety. Its singular part is $\overline{\mathcal{O}}_{[2,1^2]} = \{A \in \liesp_4 \mid A^2 = 0,\, \rk A \leq 1\}$. 
In \cite{ls:2006}, Lehn and Sorger show that the cotangent bundle $\{ (B,U) \subset \overline{\mathcal{O}}_{[2^2]} \times G \mid U \subset \ker B \} \to \overline{\mathcal{O}}_{[2^2]}$ is a symplectic resolution.

\section{The action of $Sl_2$ on $\liesl_2 \oplus \kxC^2$}
\label{sl2plusC2}


The symplectic reduction of the last representation of $Sl_2$ we consider has a more exciting geometry than the preceding quotients, which were all more or less classical. Here the singular locus admits a symplectic resolution by a blow up whose zero fibre provides the configuration for performing a Mukai flop. This gives two further symplectic resolutions which are also algebraic as their explicit construction by blowing up a certain subvariety will show. Therefore even resolutions which arise from the canonical process of blowing up are not always the only projective symplectic resolutions.
\vspace{1ex}

The action on the sum of the special linear Lie algebra and the complex space also appears in Schwarz' list for general $n$:
$$Sl_n \times (\liesl_n \oplus \kxC^n) \to \liesl_n \oplus \kxC^n, \; g \cdot (A,x) = (gAg^{-1},gx).$$
To compute the symplectic double, we consider both summands seperately. 
On $\liesl_n$, like on any semisimple Lie algebra, there is a non-degenerate bilinear form $(A,B) \mapsto \tr(AB)$, which is invariant because $\tr(gAg^{-1}gBg^{-1}) = \tr(AB)$. So $\liesl_n$ is a self-dual representation via the isomorphism
$\liesl_n \longrightarrow \liesl_n^*$, $A \longmapsto (B \mapsto \tr(AB))$.

On $\kxC^n$ we consider the elements as columns whereas $(\kxC^n)^*$ contains rows. Then evaluating a linear form $y$ at a vector $x$ corresponds to usual matrix multiplication $yx$. The computation $y(g^{-1}x) = yg^{-1}x = (yg^{-1})(x)$ shows that the dual action is $Sl_n \times (\kxC^n)^* \to (\kxC^n)^*,\; (g,y) \mapsto yg^{-1}$.
Identifying $(\kxC^n)^*$ with $\kxC^n$ by taking the transpose, the dual action writes $Sl_n \times \kxC^n \to \kxC^n,\; (g,z) \mapsto (g^{-1})^tz$.
The corresponding invariant pairing is $\kxC^n \times \kxC^n \to \kxC, (x,z) \mapsto z^tx$.

Altogether we have computed the double
$$Sl_n \acts \liesl_n \oplus \kxC^n \oplus \kxC^n \oplus \liesl_n, \; g \cdot (A,x,y,B) = (gAg^{-1},gx,(g^t)^{-1}y,gBg^{-1}).$$

First we determine the invariants of the simple action, then we combine them to invariants of the symplectic double. Taking into consideration the decomposition of the representation into a direct sum we distinguish three types of invariants: pure invariants of $\kxC^n$ and $\liesl_n$ respectively, and mixed invariants. 
Considering $\kxC^n$, the first fundamental theorem for $Sl_n$ tells us that there are no invariants. 
According to the first fundamental theorem for matrices, generating invariants of conjugation of $Sl_n$ on its Lie algebra are exactly the coefficients of the characteristic polynomial or equivalently the traces $\tr A^k, \; k = 1, \hdots, n$. Of course $\tr A = 0$, so $k = 1$ is superfluous. 
A mixed invariant is $\det (x \mid Ax \mid \hdots \mid A^{n-1}x)$, since $\det(g) = 1$ and
$\det (gx \mid gAg^{-1}gx \mid \hdots \mid (gAg^{-1})^{n-1}gx) 
= \det(g) \det (x \mid Ax \mid \hdots \mid A^{n-1}x).$ 
Specialising to the case $n = 2$ we can prove

\begin{prop}
The ring of invariants of the simple action $Sl_2 \acts \liesl_2 \oplus \kxC^2$ is
$$\kxC[\,\liesl_2 \oplus \kxC^2]^{Sl_2} = \kxC[\, \tr A^2,\; \det (x \midd Ax)] = \kxC[\, \det A,\; \det (x \midd Ax)].$$
\end{prop}
\begin{proof}
We have already seen that both elements are invariants, so we only have to show that $\tr A^2$ and $\det (x \midd Ax)$ generate the ring of invariants. The Poincaré--series (cf. \cite{muk:2003, pv:1994}) of the simple action is $\frac{1}{(1-t^2)(1-t^3)}$ . The denominator indicates that there are two generating invariants, one of degree $2$ and one of degree $3$. They are algebraically independent because the numerator is $1$.

Now denoting $A = \left(\begin{smallmatrix} a_{11} & a_{12} \\ a_{21} & -a_{11} \end{smallmatrix}\right)$ and $x = \binom{x_1}{x_2}$ the invariant $\tr A^2 = 2(a_{11}^2 + a_{12}a_{21})$ is of degree $2$ and $\det (x \midd Ax) = a_{21} x_1^2 - 2a_{11}x_1x_2 - a_{12}x_2^2$ is of degree $3$. Obviously both invariants are algebraically independent, so they generate the ring of invariants. The last equation holds because $\tr A^2 = -2\det A$.
\end{proof}

Now we extrapolate the invariants of the simple action to invariants of the doubled one. 
Let $\Omega := \{A^{a_1}B^{b_1}\cdots A^{a_l}B^{b_l} \mid a_i, b_i \in \mathbbm{N}_0 \text{ for } i = 1,\hdots,l \}$ be the set of all words with characters $A$ and $B$. 
The invariants of conjugation by $Sl_n$ on two copies of its Lie algebra are the traces
$\tr W$, $W \in \Omega$, due to the first fundamental theorem for matrices. 
On $\kxC^n \oplus \kxC^n$ the first fundamental theorem for $Sl_n$ yields the existence of exactly one generating invariant, namely $y^tx$. But there are also mixed invariants: $y^tWx$, $W \in \Omega$, are invariants, as well as
$\det (W_1 x \mid W_2 x \mid \hdots \mid W_n x)$ and $\det (W_1^t y \mid W_2^t y \mid \hdots \mid W_n^t y)$, $W_i \in \Omega$. These mixed expressions in $x, y$ and $A, B$ are invariants, because $Sl_n$ acts on each word in  $\Omega$ via conjugation of the whole word as the actions in the middle eliminate themselves.

Again we can show that there is a set of generators of the ring of invariants among the invariants we have already found if $n = 2$:
\begin{prop}
\label{verdInv}
Generating invariants of $Sl_2 \acts \liesl_2 \oplus \kxC^2 \oplus \kxC^2 \oplus \liesl_2$ are
\begin{itemize}
\item $\det A,\; \det B,\; \tr AB$,
\item $y^tx,\; y^tAx,\; y^tBx,\; y^tABx$,
\item $\det (x \mid Ax),\; \det (x \mid Bx),\; \det (x \mid ABx)$ and
\item $\det (y \mid A^ty),\; \det (y \mid B^ty),\; \det (y \mid (AB)^ty)$.
\end{itemize}
\end{prop}
\begin{proof}
The Poincaré--series of this representation is $-\frac{t^6-t^5+t^4+2t^3+t^2-t+1}{(t+1)^3(t^2+t+1)^3(t-1)^7}$, which has the expansion $1 + 4t^2 + 6t^3 + 13t^4 + 24t^5 + O(t^6)$.
The four invariants $\det A$, $\det B$, $\tr AB$, $y^tx$ are independent of degree $2$, whereas $y^tAx$, $y^tBx$, $\det (x | Ax)$, $\det (x | Bx)$, $\det(y | A^ty)$ and $\det(y | B^ty)$ are six invariants of degree $3$ and $y^tABx$, $\det (x | ABx)$, $\det (y | (AB)^ty)$ have degree $4$. None of these invariants can be expressed in terms of the others. If these are generators of the ring of invariants, we have to multiply the numerator of the Poincaré--series by $(1+t)(1-t^3)^3(1-t^4)^3$. Doing this, we obtain
$1-6t^6-8t^7-6t^8+8t^9+24t^{10}+24t^{11}+5t^{12}-24t^{13}-36t^{14}-24t^{15}+5t^{16}+24t^{17}+24t^{18}+8t^{19}-6t^{20}-8t^{21}-6t^{22}+t^{28}$
which is indeed the Hilbert--polynomial of the indicated invariants.
\end{proof}

For computing the moment map $\mu: \liesl_n \oplus \kxC^n \oplus \kxC^n \oplus \liesl_n \to \liesl_n^*$ explicitly we use a pairing comprised of the trace pairing on $\liesl_n$ and $(x,z) \mapsto z^tx$ on $\kxC^n$. Because of $\tr([A,B]) = 0$ and $\tr(xy^t) = \frac{1}{n}y^tx$ we have
$\mu(A,x,y,B)(\xi) = \tr((\xi A - A \xi)B) + y^t \xi x = \tr(\xi (AB - BA + xy^t))$. 
Using the identification $\liesl_n^* \longrightarrow \liesl_n, \, (B \mapsto \tr(AB)) \longmapsto A - \tr(A) I_n,$ we can consider $\mu$ as a map to $\liesl_n$. Then
$$\mu(A,x,y,B) = [A,B] + xy^t - \tfrac{1}{n}y^txI_n.$$

We write $V := \liesl_2 \oplus \kxC^2 \oplus \kxC^2 \oplus \liesl_2$ and $\kxC[a_{11},a_{12},a_{21},b_{11},b_{12},b_{21},x_1,x_2,y_1,y_2]$ for the ring of coordinates, so $(A,x,y,B) = 
\left( \bigl(\begin{smallmatrix} a_{11} & a_{12} \\ a_{21} & -a_{11} \end{smallmatrix}\bigr),  \binom{x_1}{x_2}, \binom{y_1}{y_2},
\bigl(\begin{smallmatrix} b_{11} & b_{12} \\ b_{21} & -b_{11} \end{smallmatrix}\bigr) \right) \in V$. 
Determining the preimage of zero under the moment map, the defining equation $0 = \mu(A,x,y,B) = [A,B] + xy^t - \frac{1}{2}y^txI_2$ yields the momentum ideal\par
$\II_{\mu} = (2 a_{11} b_{12} \mi 2 a_{12} b_{11} \pl x_1 y_2,\,a_{12} b_{21} \mi a_{21} b_{12} \pl \frac{1}{2} (x_1 y_1 \mi  x_2 y_2), 
\, 2 a_{21} b_{11} \mi 2 a_{11} b_{21} \pl x_2 y_1).$

\begin{rem} Modulo $\II_{\mu}$ the elements of $\Omega$ can be reduced to sorted words, i.e. elements in $\{A^k B^l \mid 0 \leq k,\, l \leq n-1 \}$ in the general case. As the generating invariants in the case $n = 2$ only consist of terms with sorted words this does not lead to decreasing the number of invariants here.
\end{rem}
Anyway, we need less than thirteen invariants to generate the symplectic reduction. A calculation with SINGULAR \cite{singular} shows that five of the original invariants can be omitted modulo $\II_{\mu}$, there remain eight invariants $z_1,\hdots,z_8$ with nine generating relations.

\begin{prop} The ring of invariants of the action of $Sl_2$ on $\mu^{-1}(0)$ is\par
$\kxC[\mu^{-1}(0)]^{Sl_2} = \kxC[z_1,\hdots,z_8]/(h_1,\hdots,h_9)$
with invariants
$$\begin{array}{llll}
z_1 = \det A, & z_2 = \tr AB, & z_3 = \det B, & z_4 = y^tx,\\
z_5 = \det (x \mid Ax), & z_6 = \det (x \mid Bx), & z_7 = \det (y \mid A^ty), & z_8 = \det (y \mid B^ty)
\end{array}$$
and relations $h_i = 0$, $i = 1,\hdots,9$, where
$$\begin{array}{ll}
h_1 = (2z_2-z_4)z_7 + 4z_1z_8, & h_6 = (2z_2+z_4)z_4^2 - 4z_6z_7,\\
h_2 = (2z_2+z_4)z_8 + 4z_3z_7, & h_7 = (2z_2-z_4)z_4^2 - 4z_5z_8,\\
h_3 = (2z_2+z_4)z_5 + 4z_1z_6, & h_8 = z_3z_4^2+z_6z_8,\\
h_4 = (2z_2-z_4)z_6 + 4z_3z_5, & h_9 = z_1z_4^2+z_5z_7.\\
h_5 = (2z_2+z_4)(2z_2-z_4) - 16z_1z_3,
\end{array}$$
\end{prop}


The proposition shows that our quotient $V \sred Sl_2 = \mu^{-1}(0) \red Sl_2$ is the set of zeros $Z := \mathcal{V}(h_1,\hdots,h_9)$ in $\kxC^8$. 
Arranging the invariants as a matrix
$$M := \begin{pmatrix} 2z_2-z_4 & 4z_3 & z_8 \\ 4z_1 & 2z_2+z_4 & -z_7 \\ z_5 & -z_6 & \frac{1}{4}z_4^2\end{pmatrix}$$
gives a description of the relations as its $2\times 2$--minors. So our quotient consists of all matrices $M$ as above with rank at most $1$.
\vspace{1ex}

Further SINGULAR calculations combined with Jacobi's criterion show:
\begin{prop}
The quotient $Z$ is a variety of dimension $4$, whose singular locus is $Z_{sing} := \mathcal{V}(z_8,z_7,z_6,z_5,z_4, z_2^2-4z_1z_3) \in \kxC^8$.
\end{prop}

\begin{rem}
The structure of $A_1$--singularity attracts attention at once: obviously the singular locus is $\mathcal{V}(z_2^2-4z_1z_3)$ in the subspace $\mathcal{V}(z_8,z_7,z_6,z_5,z_4) \cong \kxC^3$. In particular it is again singular in the origin. Thus $Z$ is stratified as
$Z \supset Z_{sing} \supset \{0\}$ with dimensions $4$, $2$, $0$ respectively. 
This phenomenon of each singular stratum being contained in the preceding one with even codimension is typical of symplectic varieties. Indeed $Z$ is one:
\end{rem}

\begin{prop}
\label{symplVar}
$Z$ is a symplectic variety.
\end{prop}
\begin{proof}
With the help of SINGULAR we see that $\mu^{-1}(0)$ is a normal complete intersection and we compute $\mu^{-1}(0)_{sing}\red Sl_2 = Z_{sing}$. The singular locus $\mu^{-1}(0)_{sing}$ is given by the equations $y^tx=0$, $Ax=0$, $Bx=0$, $y^tA=0$, $y^tB=0$ and $AB-BA=0$. Let $(A,x,y, B) \in \mu^{-1}(0) \setminus \mu^{-1}(0)_{sing}$. If $x = 0$ and $y = 0$, then $AB-BA = 0$ because of the defining equation $AB-BA + xy^t - \frac{1}{2}y^txI_2 = 0$ of $\mu^{-1}(0)$, and therefore $(A,x,y,B) \in \mu^{-1}(0)_{sing}$. So w.l.o.g. we can assume $x \neq 0$. Complete $x$ to a basis $(x,x')$ of $\kxC^2$. If $g \in Sl_2$ stabilises $x$, it must have the shape $g = \left(\begin{smallmatrix} 1 & s \\ 0 & t \end{smallmatrix}\right)$ with respect to the basis $(x,x')$. But $\det g = 1$ implies $t = 1$. Then for every integer $n$ we have $g^n = \left(\begin{smallmatrix} 1 & s^n \\ 0 & 1 \end{smallmatrix}\right)$. So $g$ has infinite order unless $s = 0$. But the isotropy group of $(A,x,y,B)$ is finite, because $\mu^{-1}(0)$ is a complete intersection, so the isotropy group must even be trivial. 
By proposition \ref{complint} it only remains to construct a resolution $\pi: \widetilde{Z} \to Z$ such that the symplectic form on $\pi^{-1}(Z_{reg})$ extends to $\widetilde{Z}$. This will be done next.
\end{proof}

Now we describe a resolution of $Z$ via blowing up the singular locus once. 
At first the following observation simplifies the process of blowing up enormously:
Let us once and for all identify $\kxC^3 = \mathcal{V}(z_4,z_5,z_6,z_7,z_8) \subset \kxC^8$. The fact $Z \cap \kxC^3 = Z_{sing}$ implies that the blow up of $Z$ in $Z_{sing}$ is the strict transform of the blow up $\pi\colon Bl_{\kxC^3}(\kxC^8) \to \kxC^8$.
Now by an easy computation the blow up of $Z$ in $Z_{sing}$ is
\begin{align*}
\widetilde{Z} := Bl_{Z_{sing}}(Z) = \{ ((z_1,\hdots,z_8),[&x_4:x_5:x_6:x_7:x_8]) \in \kxC^8 \times \mathbbm{P}^4 \mid \\
 &z_ix_j = z_jx_i, \; i, j = 4,\hdots,8, \; \ell_1 = \hdots = \ell_9 = 0 \}.
\end{align*}

Here the $\ell_i$ are defined by $\rk \widetilde M \leq 1$ and $\widetilde M$ is obtained by substituting $z_i$ by the corresponding $x_i$ in the last row and column of $M$.

\begin{prop}
The blow up $\pi\colon \widetilde{Z} \to Z$ is a resolution of singularities.
\end{prop}

\begin{proof}
Let us exemplarily deal with the chart $x_4 = 1$. The strict transform is the set of zeros of the polynomials
$$\begin{array}{ll}
(2z_2-z_4)x_7 + 4z_1x_8, & (2z_2+z_4) - 4x_6x_7,\\
(2z_2+z_4)x_8 + 4z_3x_7, & (2z_2-z_4) - 4x_5x_8,\\
(2z_2+z_4)x_5 + 4z_1x_6, & z_3+x_6x_8,\\
(2z_2-z_4)x_6 + 4z_3x_5, & z_1+x_5x_7.\\
(2z_2+z_4)(2z_2-z_4) - 16z_1z_3,
\end{array}$$
The four equations on the right give $z_1 = -x_5x_7$, $z_2 = x_6x_7 + x_5x_8$, $z_3 = -x_6x_8$, $z_4 = 2(x_6x_7 - x_5x_8).$
Inserting this into the five remaining equations yields
$$\begin{array}{ll}
(4x_5x_8)x_7 + 4(-x_5x_7)x_8 = 0, & (4x_6x_7)x_5 + 4(-x_5x_7)x_6 = 0,\\
(4x_6x_7)x_8 + 4(-x_6x_8)x_7 = 0, & (4x_5x_8)x_6 + 4(-x_6x_8)x_5 = 0,\\
(4x_6x_7)(4x_5x_8) - 16(-x_5x_7)(-x_6x_8) = 0.
\end{array}$$
Thus the strict transform is smooth, and the same holds for $\widetilde{Z}$ because one can treat the other charts $x_5 = 1,\; x_6 = 1,\; x_7 = 1$ and $x_8 = 1$ in a similar way.
\end{proof}

Our next interest concerns the fibres of the resolution, because they indicate if $\pi$ is semismall. This is a property needed to prove symplecticity.
\begin{prop}
\label{fasern1}
The fibre $\pi^{-1}(z)$ of the resolution $\pi: \widetilde{Z} \to Z$ is
\begin{itemize}
\item a point if $z \in Z \setminus Z_{sing}$,
\item isomorphic to $\prP^1$ if $z \in Z_{sing} \setminus \{0\}$,
\item the union of two projective planes $E_1 \cup E_2$ intersecting in one point if $z = 0$.
\end{itemize}
\end{prop} 
\begin{proof}
Let $z = (z_1, \hdots, z_8) \in Z$. If $z \not\in Z_{sing} = Z \cap \kxC^3$ then $z_j \neq 0$ for one $j \in \{4,\hdots,8\}$. So we have $x_i = \frac{z_i}{z_j}x_j$ in the preimage, which is therefore fixed to be $\pi^{-1}(z) = \{((z_1, \hdots, z_8),[z_4:z_5:z_6:z_7:z_8])\}$.

Let now $z \in Z_{sing}$. Then we have $z_4 = z_5 = z_6 = z_7 = z_8 = 0$ and $z_2^2 = z_1z_3$.
The relations $\ell_i = 0, \; i = 1, \hdots, 9$, reduce to
$\ell_1 = z_2x_7 \pl 2z_1x_8$, $\ell_2 = z_2x_8 \pl 2z_3x_7$, $\ell_3 = z_2x_5 \pl 2z_1x_6$, $\ell_4 = z_2x_6 \pl 2z_3x_5$, $\ell_5 = 0$, $\ell_6 = z_2x_4^2 \mi 2x_6x_7$, $\ell_7 = z_2x_4^2 \mi 2x_5x_8$, $\ell_8 = z_3x_4^2\pl x_6x_8$, $\ell_9 = z_1x_4^2\pl x_5x_7$.

To begin, we consider the case $z_2 \neq 0$. 
If $x_4 \neq 0$, the equations $\ell_6 = \hdots = \ell_9 = 0$ yield
$z_2 = 2x_6x_7 = 2x_5x_8$, $z_3= -x_6x_8$, $z_1 = -x_5x_7$. 
Inserting this, $\ell_1, \hdots, \ell_4$ are automatically fulfilled.
The assumption $z_2 \neq 0$ implies $x_i \neq 0$ for $i = 5,\, 6,\, 7,\, 8$. Setting $x_8 =: t =: \frac{a}{b} \in \kxC^*$, we obtain $ x_5 = \frac{z_2}{2t}$, $x_6 = \frac{-z_3}{t}$, $x_7 = \frac{-z_1}{x_5} = \frac{-2z_1}{z_2}t$, so 
$[x_4:x_5:x_6:x_7:x_8] = 
[t:\frac{z_2}{2}:-z_3:\frac{-2z_1}{z_2}t^2:t^2] = [ab:\frac{z_2}{2}b^2:-z_3b^2:\frac{-2z_1}{z_2}a^2:a^2]$.

If $x_4 = 0$, then $\ell_6, \hdots, \ell_9$ mean $x_5 = 0 = x_6$ or $x_7 = 0 = x_8$. In the first case $\ell_1$ yields $x_7 = -\frac{2z_1}{z_2}x_8$. Setting $t = x_8$ we obtain $[0:0:0:\frac{-2z_1}{z_2}t:t]$, which is exactly the above element in the case $b = 0$. Alike, in the second case $\ell_3$ yields $x_5 = -\frac{2z_1}{z_6}x_6$. Setting $x_6 = -z_3s$ leads to $x_5 = \frac{2z_1z_3}{z_6}s = \frac{z_2}{2}s$, which is again compatible with the above element for $s = \frac{b}{a}$. 
Altogether the fibre is parametrised by $\prP^1$:
$$\pi^{-1}(z) = \{((z_1, \hdots, z_8),[ab:\tfrac{z_2}{2}b^2:-z_3b^2:\tfrac{-2z_1}{z_2}a^2:a^2]) \mid (a,b) \in \kxC^2 \setminus\{(0,0)\} \}.$$

In the case $z_2 = 0$ the condition $z_2^2 = 4z_1z_3$ yields $z_1 = 0$ or $z_3 = 0$. 
If $z_1 \neq 0$ and $x_4 = 1$ then $z_1 = -x_5x_7$ implies $x_5 \neq 0$ and $x_7 \neq 0$. Thus the conditions on $z_1$ and $z_2$ become $x_6 = x_8 = 0$ and $x_7 = -\frac{z_1}{x_5}$, so with $t = x_5$ we obtain $[1:t:0:-\frac{z_1}{t}:0] = [ab:a^2:0:-z_1b^2:0]$. 
If $x_4 = 0$ then $\ell_1$ yields $x_8 = 0$, $\ell_3$ affords $x_6 = 0$ and $\ell_9$ implies $x_5 = 0$ or $x_7 = 0$.
On the whole the fibre
$$\pi^{-1}(z) = \{((z_1, \hdots, z_8),[ab:a^2:0:-z_1b^2:0]) \mid (a,b) \in \kxC^2 \setminus\{(0,0)\} \},$$
is as well isomorphic to $\prP^1$. 
Analogously in the case $z_3 \neq 0$
we obtain the fibre
$$\pi^{-1}(z) = \{((z_1, \hdots, z_8),[ab:0:-z_3b^2:0:a^2]) \mid (a,b) \in \kxC^2 \setminus\{(0,0)\} \},$$
again parametrised by $\prP^1$. 
This proves the second part of the assertion.

Finally we deal with the case $z = 0$. Here the equations $\ell_1,\hdots,\ell_5$ do not provide any information, whereas  $\ell_6,\hdots,\ell_9$ yield $ x_6x_7 = 0$, $x_5x_8 = 0$, $x_6x_8 = 0$, $x_5x_7 = 0$, equivalently as equation of ideals $(x_5,x_6)(x_7,x_8) = 0.$ 
Each of the two ideals defines a projective plane $E_1 := \{[x_4:0:0:x_7:x_8]\}$ resp. $E_2 := \{[x_4:x_5:x_6:0:0]\}$, whose union constitutes the zero fibre $\pi^{-1}(0)$. They intersect in the single point $[1:0:0:0:0]$.
\end{proof}

\begin{prop}
The resolution $\pi: \widetilde{Z} \to Z$ is semismall.
\end{prop}
\begin{proof} According to the computations of the fibres we have
\begin{itemize}
\item $x \in Z\setminus Z_{sing}: \quad \dim \pi^{-1}(x) = 0 = \frac{1}{2}\codim Z$,
\item $x \in Z_{sing} \setminus \{0\}: \quad \dim \pi^{-1}(x) = 1 = \frac{1}{2}\codim Z_{sing}$,
\item $x = 0: \quad \dim \pi^{-1}(0) = 2 = \frac{1}{2}\codim\{0\}$,
\end{itemize}
so for each stratum $Z \supset Z_{sing} \supset \{0\}$ the defining inequality for semismallness is satisfied.
\end{proof}

\begin{prop}
The resolution $\pi: \widetilde{Z} \to Z$ is symplectic.
\end{prop}
\begin{proof}
It is well-known that any symplectic form extends to the blow up of an $A_1$--singularity, so the symplectic form of $Z$ extends to the preimage of $Z_{sing}\setminus\{0\}$. Thus it is not defined at most over the stratum $\{0\}$ of codimension $4$. But as the resolution is semismall proposition \ref{cohalbsympl} can be applied, which says that the symplectic form extends symplectically everywhere.
\end{proof}

As projective planes the two components $E_1$ and $E_2$ of the zero fibre constitute a configuration admitting a Mukai flop in the four dimensional variety $\widetilde{Z}$.

\begin{mind}
Let $(X,\sigma)$ be a symplectic manifold of dimension $2n > 2$, $P \subset X$ a closed submanifold isomorphic to $\prP^n$. Let further $f\colon Z \to X$ be the blow up of $X$ in $P$ and let $D \subset Z$ denote the exceptional divisor. Then $D \to P$ is isomorphic to the incidence variety $\{(l,H) \subset \prP^n \times (\prP^n)^* \mid l \in H\}$. There exists a blow down $g\colon Z \to X'$ such that $D$ is the exceptional divisor and the restriction of $g$ to $D$ is the projection $D \to (\prP^n)^*$. The variety $\elm_P(X) := X'$ is also symplectic. It is called \emph{Mukai's elementary transform} or \emph{Mukai flop} of $X$. If $X$ is projective, $X'$ need not necessarily be projective.
\end{mind}

Locally we are exactly in the situation of Fu's and Namikawa's example $5$ in \cite{funa:2004}, so effecting a  Mukai flop at $E_1$ or $E_2$ gives a further symplectic resolution of $Z$. 
Writing $X := \elm_{E_2}(\widetilde{Z})$ we obtain the following configuration, where the first diagram visualises the blow ups and the second one describes the restrictions to the zero fibre of $\pi$: \vspace{-1ex}
$$ \begin{xy} \xymatrix{
 & \Bl_{E_2}\widetilde{Z} \ar@{->}_f[ld] \ar@{->}^g[rd] & \\
\widetilde{Z} \ar@{->}_{\pi}[d] & & X\\
Z &  &
}\end{xy}
\qquad
\begin{xy} \xymatrix{
 & \widetilde{E_1}\cup F \ar@{->}_f[ld] \ar@{->}^g[rd] & \\
E_1 \cup E_2 \ar@{->}_{\pi}[d] & & \widetilde{E_1}\cup E_2^*\\
0 &  &
}\end{xy} $$
Here $F$ is the incidence variety $\{(l,H) \in E_2 \times E_2^*\mid l \in H\}$. As $E_1$ and $E_2$ intersect in one point, $E_1$ does not remain untouched under the blow up. We denote the strict transform of $E_1$  by $\widetilde{E_1}$. The point of intersection $E_1 \cap E_2$ becomes a line of intersection $\widetilde{E_1} \cap F$, likewise $\widetilde{E_1}$ and $E_2^*$ intersect in a line.
\vspace{1em}

Our idea for constructing $X \to Z$ explicitly and for showing that $X$ is even algebraic is to consider a partial resolution of $Z$. Instead of resolving $Z_{sing}$ we blow up $\kxC^8$ in $\mathcal{V}(z_4,z_7,z_8) \cong \kxC^5$ and take the strict transform. This corresponds to the blow up of $Z$ in $S = \mathcal{V}(z_4,z_7,z_8, z_2^2 - 4z_1z_3)$, which we denote by $Y$:
\begin{align*}
Y = \{ ((z_1,\hdots,z_8),\,[y_4:y_7:y_8]) &\in \kxC^8 \times \mathbbm{P}^2 \mid \\
&z_iy_j = z_jy_i, \; i, j = 4,7,8, \; k_1 = \hdots = k_9 = 0\}
\end{align*}
where the $k_i$ are defined by 
$\rk M_2 \leq 1$, $M_2 = \left( \begin{smallmatrix} 2z_2-z_4 & 4z_3 & y_8 \\ 4z_1 & 2z_2+z_4 & -y_7 \\ z_5 & -z_6 & \frac{1}{4}z_4y_4\end{smallmatrix} \right)$.
\vspace{0.5ex}

Analysing the fibres of $\pi_2\colon Y \to Z$ we show that $Y$ resolves the $A_1$--singularities. Thus only an isolated singularity in the origin remains.
\begin{prop} The fibre $\pi_2^{-1}(z)$ is
\begin{itemize}
\item one point if $z \in Z \setminus \kxC^3$,
\item isomorphic to $\prP^1$ if $z \in \kxC^3 \setminus \{0\}$,
\item $E_1$ if $z = 0$.
\end{itemize}
\end{prop}
\begin{proof}
At first let $z \in Z \setminus \kxC^5$. Then $z_4 \neq 0,\;z_7 \neq 0$ or $z_8 \neq 0$, so the conditions $z_iy_j = z_jy_i$ determine $y = [y_4:y_7:y_8]$ uniquely. Furthermore $(z,y)$ satisfies $k_1,\dots,k_9$ so that the preimage is not empty.\par
Let now $z \in \kxC^5 \setminus \kxC^3$. Then $z_4 = z_7 = z_8 = 0$, but $z_5 \neq 0$ or $z_6 \neq 0$. Thus $k_6$, $k_7$, $k_8$ and $k_9$ simplify to $z_6y_7 = 0$, $z_5y_8 = 0$, $z_6y_8 = 0$, $z_5y_7 = 0$, so we have $y_7 = 0 = y_8$. This makes $k_1, k_2$ being satisfied automatically, whereas $k_3, k_4, k_5$ hold because $z \in Z$. So the preimage $\pi_2^{-1}(z) = \{(z,[1:0:0])\}$ is exactly one point.\par
In the case $z \in \kxC^3 \setminus \{0\}$ we have $z_4 = \hdots = z_8 = 0$, therefore the equations are $z_2y_7 + 2z_1y_8 = 0$, $z_2y_7 + 2z_3y_7 = 0$ and $z_2^2 - 4z_1z_3 = 0$. 
If $z_2 \neq 0$ we have $y_7 = \frac{-2z_1}{z_2}y_8$ or equivalently $y_8 = \frac{-2z_3}{z_2}y_7$. Thus each $y$ in the preimage is of the type $y = [z_2b:-2z_1a:z_2a] = [z_2b:z_2a:-2z_3a]$. 
If $z_2 = 0$ it follows $z_1 = 0$, $z_3 \neq 0$ or conversely. In the first case we conclude $y_7 = 0$, $y_8$ arbitrary, in the second case $y_8 = 0$, $y_7$ arbitrary, both cases being of the type computed above. As $(a, b) \in \kxC\setminus\{(0,0)\}$ can be chosen arbitrarily, the preimage is $\prP^1$.\par
Finally, the case $z = 0$ does not put any condition on the homogeneous coordinate $y$, so $\pi_2^{-1}(0) = \{(0,[y_4:y_7:y_8])\}$ is isomorphic to the projective plane $E_1$.
\end{proof}

Now we want to resolve the remaining singularity. As the remaining singularity is of codimension $4$, a further blow up would attach a $\prP^3$ and the dimension of the zero fibre would be $3 > 2 = \frac{1}{2}\codim \{0\}$, such that the map would not be semismall and neither symplectic.\par
Locally in the chart $y_4 = 1$ the variables $z_7 = z_4y_7$ and $z_8 = z_4y_8$ can be omitted and $z_7y_8 = z_4y_7y_8 = y_7z_8$ is automatic. So $M_2$ yields the local description
$$\left \{(z_1,z_2,z_3,z_4,z_5,z_6,y_7,y_8)\, \middle|\; \rk \left(\begin{smallmatrix} -2z_2+z_4&4z_3&-8y_8\\ -4z_1&2z_2+z_4&8y_7\\-z_5&-z_6&-2z_4 \end{smallmatrix}\right)\leq 1\right\}.$$
A linear coordinate transformation is enough to transform this into the variety
$\{A \in \liesl_3\mid \rk A \leq 1\}$
which has two well-known symplectic resolutions
$$\begin{array}{l}
\sigma_1\colon \widetilde{Y}':= \{(A,L) \in Y \times \prP^2 \mid \im A \subset L\} \to Y,\\
\sigma_2\colon \widetilde{Y} := \{(A,H) \in Y \times (\prP^2)^* \mid H \subset \ker A\} \to Y.
\end{array}$$
The compositions $\pi_2 \circ \sigma_1$ and $\tau_2 = \pi_2 \circ \sigma_2$ are symplectic resolutions of $Z$.
\vspace{1em}

Now we can examine our original question, namely if $X$ coincides with one of these resolutions and $\widetilde Z$ with the other one. To obtain a positive anwer we need the existence of a map $\widetilde Z \to Y$:\hspace{-1em}
$$ \begin{xy} \xymatrix{
\widetilde{Z} \ar@{->}^{\pi}[dr] \ar@{.>}[r] & Bl_S(Z) = Y \ar@{->}^{\pi_2}[d]<-2mm> \\
& \quad Z \supset S
}\end{xy} $$
By a SINGULAR calculation $\codim (\pi^{-1}\II_S)\mathcal{O}_{\widetilde Z} = 1$, so $S$ defines a Weil divisor and even a Cartier divisor in $\widetilde Z$, because in our case Weil and Cartier divisors coincide. Thus by the universal property of blowing up there exists a unique map $f: \widetilde Z \to Y$ with $\pi_2 \circ f = \pi$, so $\widetilde Z$ must coincide with the first resolution $\widetilde{Y}'$ by the fact that $Y$ admits only two symplectic resolutions. 
The other resolution of $Y$ is a Mukai flop of the first one, so we obtain $X = \widetilde{Y}$. 
Blowing up $Z$ in $\mathcal{V}(z_4, z_5, z_6)$ yields a variety $Y_1$ isomorphic to $Y_2 := Y$ with a morphism $\pi_1 \colon Y_1 \to Z$, and two symplectic resolutions $\widetilde{Y}_1' \cong \widetilde Z$, $\tau_1 \colon \widetilde{Y}_1 \mapsto Z$, where $\widetilde{Y}_1$ is the Mukai flop obtained by flopping $E_1$. Altogether we have found three projective symplectic resolutions $\widetilde{Y}_1$, $\widetilde{Y}_2$ and $\widetilde Z$:
By construction, $\widetilde{Y}_1$ and $\widetilde{Y}_2$ are isomorphic as varieties but non-isomorphic as resolutions of $Z$. However, they are equivalent in the sense of \cite{funa:2004}: Let $\varphi$ be the automorphism of $Z$ which sends $z_1, z_2, z_3, z_4$ to themselves and which interchanges $z_5$ with $z_7$ and $z_6$ with $z_8$, respectively. Then $\pi_2^{-1} \circ \varphi \circ \pi_1$ is an isomorphism of $Y_1$ and $Y_2$. This induces an
isomorphism $\tau_2^{-1} \circ \varphi \circ \tau_1$ of the resolutions $\widetilde{Y_1}$ and $\widetilde{Y_2}$. 
On the contrary, due to loc. cit. $\widetilde Z$ is not equivalent to them.

\end{document}